\documentclass[12pt]{amsart}  

\usepackage[latin1]{inputenc}
\usepackage{amsmath} 
\usepackage{enumerate}
\usepackage{amsfonts}
\usepackage{amssymb}
\usepackage{stmaryrd}
\usepackage{latexsym} 
\usepackage{graphicx}
\usepackage{subfigure}
\usepackage{color}
\usepackage{hyperref}
\usepackage{verbatim}
\usepackage[all]{xy}
\usepackage{graphics}
\usepackage{pdfsync}
\usepackage[normalem]{ulem}
\usepackage{tikz}
\usepackage{url}
\usepackage{young}
\usepackage{youngtab}

\theoremstyle{plain}
\newtheorem{theorem}{Theorem}[section]
\newtheorem{proposition}[theorem]{Proposition}

\newtheorem{lemma}[theorem]{Lemma}
\newtheorem{corollary}[theorem]{Corollary}
\newtheorem{conjecture}[theorem]{Conjecture}

\newtheorem{example}[theorem]{Example}

\theoremstyle{remark}
\newtheorem{remark}[theorem]{Remark}

\numberwithin{equation}{section}
 
\definecolor{vividviolet}{rgb}{0.62, 0.0, 1.0}

\newcommand{\fS}{{\mathfrak S}}
\newcommand{\cS}{{\mathcal S}}
\newcommand{\cE}{{\mathcal E}}
\newcommand{\cP}{{\mathcal P}}
\newcommand{\cT}{{\mathcal T}}
\newcommand{\cA}{{\mathcal A}}
\newcommand{\cB}{{\mathcal B}}

\newcommand{\svw}{\textcolor{red}}

\newcommand{\Q}{\mathbb{Q}}

\newcommand{\Z}{\mathbb{Z}}

\newlength\cellsize 
\savebox2{%
\begin{picture}(15,15)
\put(0,0){\line(1,0){15}}
\put(0,0){\line(0,1){15}}
\put(15,0){\line(0,1){15}}
\put(0,15){\line(1,0){15}}
\end{picture}}
\newcommand\cellify[1]{\def\thearg{#1}\def\nothing{}%
\ifx\thearg\nothing
\vrule width0pt height\cellsize depth0pt\else
\hbox to 0pt{\usebox2\hss}\fi%
\vbox to 15\unitlength{
\vss
\hbox to 15\unitlength{\hss$#1$\hss}
\vss}}
\newcommand\tableau[1]{\vtop{\let\\=\cr
\setlength\baselineskip{-16000pt}
\setlength\lineskiplimit{16000pt}
\setlength\lineskip{0pt}
\halign{&\cellify{##}\cr#1\crcr}}}
\savebox3{%
\begin{picture}(15,15)
\put(0,0){\line(1,0){15}}
\put(0,0){\line(0,1){15}}
\put(15,0){\line(0,1){15}}
\put(0,15){\line(1,0){15}}
\end{picture}}
\newcommand\expath[1]{%
\hbox to 0pt{\usebox3\hss}%
\vbox to 15\unitlength{
\vss
\hbox to 15\unitlength{\hss$#1$\hss}
\vss}}
\newcommand\bas[1]{\omit \vbox to \cellsize{ \vss \hbox to \cellsize{\hss$#1$\hss} \vss}}

\begin{document}

\title[Chromatic Posets]{Chromatic Posets}

\author{Samantha Dahlberg}
\address{
School of Mathematical and Statistical Sciences,
Arizona State University,
Tempe, AZ 85287, USA}
\email{sdahlber@asu.edu}

\author{Adrian She}
\address{
Department of Computer Science,
University of Toronto,
Toronto, ON M5S 2E4, Canada}
\email{ashe@cs.toronto.edu}

\author{Stephanie van Willigenburg}
\address{
 Department of Mathematics,
 University of British Columbia,
 Vancouver, BC V6T 1Z2, Canada}
\email{steph@math.ubc.ca}

\thanks{All authors were supported  in part by the National Sciences and Engineering Research Council of Canada.}

\begin{abstract}
In 1995 Stanley introduced the chromatic symmetric function $X_G$ of a graph $G$, whose $e$-positivity and Schur-positivity has been of large interest. In this paper we study the relative $e$-positivity and Schur-positivity between connected graphs on $n$ vertices. We define and investigate two families of posets on distinct chromatic symmetric functions. The relations depend on the $e$-positivity or Schur-positivity of a weighed subtraction between $X_G$ and $X_H$. We find a biconditional condition between $e$-positivity or Schur-positivity and the relation to the complete graph. This gives a new paradigm for  $e$-positivity and for  Schur-positivity. We show many other interesting properties of these posets including that trees  form an independent set and are maximal elements. Additionally, we find that stars are independent elements,  the independence number increases as we increase  in the poset and that the family of lollipop graphs form a chain. 
\end{abstract}

\maketitle
\tableofcontents

\section{Introduction}

The chromatic symmetric function defined by Stanley~\cite{Stan95} is a generalization  of the chromatic polynomial by Birkhoff~\cite{B12} that has received a lot of attention lately. Many properties of the chromatic polynomial are generalized by the chromatic symmetric function including number of acyclic orientations~\cite[Theorem 3.3]{Stan95}, but not the property of deletion-contraction. The study of these symmetric functions has taken many directions. One direction studies which graphs are distinguished by their chromatic symmetric function or not~\cite{Orellana}. Though all trees have the same chromatic polynomial Stanley~\cite[p 170]{Stan95}  conjectures that non-isomorphic trees are distinguished, which has been studied but not fully resolved~\cite{Jose2+1, Jose2, MMW, Orellana}. 
Due to its connections with representation theory and algebraic geometry another direction  has revolved around the ability to write a chromatic symmetric function as a non-negative linear combination of elementary symmetric functions or Schur symmetric functions,  properties called $e$-positivity~\cite{ChoHuh,  DFvW, DvW, Gash, GebSag, GP,    HP, Wolfe} and Schur-positivity~\cite{Gasharov, SW,  SWW} respectively. A conjecture by Stanley and Stembridge connected to  immanants of  Jacobi-Trudi matrices~\cite{StanStem} has brought attention to a particular family of graphs, unit interval graphs. In this paper we particularly consider the family of lollipop graphs, which are unit interval graphs and encapsulate the families of path and  complete graphs. Lollipop graphs  have been proven to be $e$-positive~\cite{GebSag} and have descriptive formulas~\cite{DvW, Wolfe}. They are important in the study of random walks~\cite{BW, Feige, Jonasson}. 

In this paper we consider ${\mathcal G}_n$ the set of equivalence classes of connected graph on $n$ vertices determined by distinct chromatic symmetric functions. We form two posets on ${\mathcal G}_n$ with cover relations determined by the relative $e$-positivity and Schur-positivity between distinct chromatic symmetric functions. We find that this poset has many interesting properties including that the complete graph is a minimal element of the poset and that trees are maximal elements. We find that a graph is $e$-positive or Schur-positive, respectively depending on the poset considered, if and only that graph is weakly greater than the complete graph. The paper is organized as follows. In Section~\ref{sec:background} we introduce chromatic symmetric functions and our  posets of interest. In particular, we show that differences of chromatic symmetric functions are neither $e$-positive nor Schur-positive in Theorems~\ref{thm:elem_X_G-X_H} and \ref{thm:Schur_X_G-X_H}, respectively. Also we prove a biconditional condition between  $e$-positivity or Schur-positivity and a property of our poset, that is its relation to the complete graph in Theorem~\ref{thm:G>K_n} and Theorem~\ref{thm:G>K_n_schur}. Interestingly we find that a  relation in one poset does not   imply a   relation in the another in general  in Proposition~\ref{prop:eisnots}. Section~\ref{sec:elem}  discusses the poset related to the elementary basis and there we prove that the independence number increases as elements increase in the poset, trees form an anti-chain, trees are maximal elements, stars are independent elements,  the poset is not a lattice and  lollipops form a chain. In Section~\ref{sec:schur} we prove analogous results for the poset related to the Schur basis, however, the proof of lollipops being a chain is more intricate. It is instead presented in Section~\ref{sec:schur_lollipop} with the analogous theorem given in Theorem~\ref{the:lollipop}. 

\section{Background}
\label{sec:background}

In this section we cover a lot of the background material needed in the rest of the paper including graphs, symmetric functions, posets and formulas for chromatic symmetric functions with sources provided for more details. After this material we will define our objects of interest,  two families of posets that investigate the relative $e$-positivity or Schur-positivity between graphs as well as the motivation behind the relations we define. We also prove a  biconditional condition between $e$-positivity or Schur-positivity and a property of the poset. 

A {\it graph} is a collection of vertices $V(G)$ and edges $E(G)$ between pairs of vertices. 
All throughout this paper when referring to a graph $G$ we will be referring to simple graphs without multi-edges or loops. If not specified otherwise $G$ is  assumed to be a \emph{connected} graph. There are a few families of graphs that we particularly refer to in this paper.  The {\it complete graph}, $K_n$ for $n\geq 1$, will have $n$ vertices and all possible edges between all pairs of vertices. The {\it path graph}, $P_n$ for $n\geq 1$, will be a graph on $n$ vertices labeled by $[n]=\{1,2,\ldots, n\}$ with edges between labels $i$ and $i+1$. Note that $P_1=K_1$. The {\it cycle graph}, $C_n$  for $n\geq 3$, will be the path graph with the additional edge between $1$ and $n$.  The {\it star graph}, $S_n$  for $n\geq 4$, will have $n$ vertices labeled by $[n]$ with edges between  $n$ and $j$ for all $j\in [n-1]$. The {\it lollipop graph}, $L_{m,n}$ for $m\geq1, n\geq 0$, will be a graph on $m+n=N$ vertices labeled with $[N]$ and will have a complete graph on $[m]$ and edges between $i$ and  $i+1$ for $i\in[m,N-1]=\{m,m+1,\ldots, N-1\}$. 
See Figures~\ref{fig:graphs} and~\ref{fig:poset}  for examples.  

\begin{figure}
\begin{tikzpicture}
\coordinate (A) at (0,0);
\coordinate (B) at (1,0);
\coordinate (C) at (1,1);
\coordinate (D) at (0,1);
\draw[black] (A)--(B)--(C)--(D)--(A)--(C);
\draw[black] (B)--(D);
\filldraw[black] (A) circle [radius=2pt] node[left] {$1$};
\filldraw[black] (B) circle [radius=2pt] node[right] {$2$};
\filldraw[black] (C) circle [radius=2pt] node[right] {$3$};
\filldraw[black] (D) circle [radius=2pt] node[left] {$4$};
\begin{scope}[shift={(3,.5)}]
\coordinate (A) at (0,0);
\coordinate (B) at (1,0);
\coordinate (C) at (2,0);
\coordinate (D) at (3,0);
\draw[black] (A)--(B)--(C)--(D);
\filldraw[black] (A) circle [radius=2pt] node[below] {$1$};
\filldraw[black] (B) circle [radius=2pt] node[below] {$2$};
\filldraw[black] (C) circle [radius=2pt] node[below] {$3$};
\filldraw[black] (D) circle [radius=2pt] node[below] {$4$};
\end{scope}
\begin{scope}[shift={(8,0)}]
\coordinate (A) at (0,0);
\coordinate (B) at (1,0);
\coordinate (C) at (1,1);
\coordinate (D) at (0,1);
\draw[black] (A)--(B)--(C)--(D)--(A);
\filldraw[black] (A) circle [radius=2pt] node[left] {$1$};
\filldraw[black] (B) circle [radius=2pt] node[right] {$2$};
\filldraw[black] (C) circle [radius=2pt] node[right] {$3$};
\filldraw[black] (D) circle [radius=2pt] node[left] {$4$};
\end{scope}
\begin{scope}[shift={(11,.5)}]
\coordinate (A) at (0,.5);
\coordinate (B) at (0,-.5);
\coordinate (C) at (2,0);
\coordinate (D) at (1,0);
\draw[black] (C)--(D)--(A);
\draw[black] (D)--(B);
\filldraw[black] (A) circle [radius=2pt] node[left] {$1$};
\filldraw[black] (B) circle [radius=2pt] node[left] {$2$};
\filldraw[black] (C) circle [radius=2pt] node[below] {$3$};
\filldraw[black] (D) circle [radius=2pt] node[below] {$4$};
\end{scope}
\end{tikzpicture}
\caption{From left to right we have $K_4$, $P_4$, $C_4$ and $S_4$.}
\label{fig:graphs}
\end{figure}
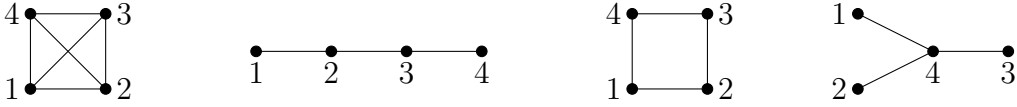

When restricting $G$ to a subset of vertices $W\subseteq V(G)$ we are referring to the graph on vertices $W$ with all edges that $G$ has between vertices in $W$. The {\it independence number} of a graph $G$, $\alpha(G)$,  is the maximal size of a subset $W\subseteq V(G)$ where $G$ restricted to $W$ has no edges. The {\it clique number} of a graph $G$, $\omega(G)$,  is the maximal size of a subset $W\subseteq V(G)$ where $G$ restricted to $W$ is a complete graph. When restricting $G$ to a subset of edges $F\subseteq E(G)$ we are referring to the graph $G$, but with the smaller edge set $F$. 

In order to define the chromatic symmetric function we will need to define proper coloring. 
A {\it proper coloring} of a graph $G$ is a map from the vertices $V(G)=\{v_1,v_2,\ldots, v_n\}$ to colors $\Z^+=\{1,2,3,\ldots\}$,
$$\kappa:V(G)\rightarrow \Z^+,$$
so that if $\epsilon\in E(G)$ is an edge between vertices $u , v \in V(G)$ then $\kappa(u)\neq \kappa(v)$. 
The {\it chromatic symmetric function} is 
$$X_G=\sum_{\kappa}x_{\kappa(v_1)}x_{\kappa(v_2)}\cdots x_{\kappa(v_n)}$$
summed over  all proper colorings $\kappa$ of $G$. 
The chromatic symmetric function generalizes the {\it chromatic polynomial} $\chi_G(k)$ that counts the number of proper colorings possible for $G$ using at most $k$ colors. Stanley~\cite[Proposition 2.2]{Stan95} showed that $X_G$ is indeed a generalization of $\chi_G(k)$ because 
$$X_G(1^k)=\chi_G(k)$$
where $X_G(1^k)$ means that we substitute in $1$ for any $k$ distinct variables and zero for the others. The chromatic polynomial generalizes the {\it chromatic number}, $\chi(G)$, the minimum number of colors required for a proper coloring.

The chromatic polynomial satisfies a very useful   deletion-contraction property, though $X_G$ does not.  Given a graph $G$ and  an edge $\epsilon\in E(G)$ the {\it deletion} of  $\epsilon$, $G-\epsilon$, is the graph $G$ with edge $\epsilon$ removed. The {\it contraction} along $\epsilon$, $G/\epsilon$, is the graph $G$ but with the two vertices on $\epsilon$ merged with any multi-edges formed merged into a single edge and loops removed.  The {\it deletion-contraction property} is that for any graph $G$ and edge $\epsilon\in E(G)$ we have that
\begin{equation}
\chi_G(k)=\chi_{G-\epsilon}(k)-\chi_{G/\epsilon}(k).
\label{eq:deletion-contraction}
\end{equation}

The chromatic symmetric functions exist inside the algebra of symmetric functions, which is a subalgebra of $\Q[[x_1,x_2,\ldots]]$ in commuting variables  where all  bases are indexed by integer partitions. An {\it integer partition},
$\lambda=(\lambda_1,\lambda_2,\ldots, \lambda_{l(\lambda)})$, is a list of weakly decreasing positive integers $\lambda_i$ called {\it parts} whose {\it length}, $l(\lambda)$, is the number of parts. If the sum of all $\lambda_i$ is $n$ we say that $\lambda$ partitions $n$, denoted by $\lambda\vdash n$.  At times we will write  $\lambda=(1^{r_1},2^{r_2},\ldots, n^{r_n})$ where $r_i$ means that $\lambda$ has $r_i$ parts of size $i$. The {\it  transpose} of $\lambda$, denoted by $\lambda^t$, is given by $\lambda^t =(r_1 +\cdots+r_n,r_2 +\cdots+r_n,\ldots,r_n)$ with zeros removed. Several change of bases formulas require the notion of dominance order. Given $\lambda, \nu\vdash n$ we say $\lambda$  {\it dominates}  $\nu$, $\lambda\succeq\nu$, if the sum of the first $j$ largest parts of $\lambda$ is always at least the sum of the first $j$ largest parts of $\nu$. 
\begin{example}\label{ex:dom}
We have that $(4,2,2)\succeq (3,3,2)$ but $(4,2,2)\not\succeq (4,3,1)$. 
\end{example}

To define the symmetric functions we define the {\it $i$-th elementary symmetric function}, which is
$$e_i=\sum_{j_1<j_2<\cdots<j_i}x_{j_1}x_{j_2}\cdots x_{j_i}$$
and the {\it elementary symmetric function} associated to $\lambda$ is 
$$e_{\lambda}=e_{\lambda_1}e_{\lambda_2}\cdots e_{\lambda_{l(\lambda)}}.$$
\begin{example}\label{ex:e}
$e_{(2,1)}= (x_1x_2+x_1x_3+x_2x_3+\cdots)(x_1+x_2+x_3+\cdots) $
\end{example}
The {\it algebra of symmetric functions}, $\Lambda$, is  the graded algebra
$$\Lambda = \Lambda_0\oplus\Lambda_1\oplus\cdots$$
where $\Lambda_0=\text{span}\{1=e_0\}=\Q$ and for $n\geq 1$
$$\Lambda_n =\text{span}\{e_{\lambda} :\lambda\vdash n\}.$$
Besides the basis of elementary symmetric functions, another classical basis that is of particular interest to us is the Schur symmetric functions. The {\it Schur symmetric function} associated to $\lambda$ can be defined using a Jacobi-Trudi identity
$$s_{\lambda} =\text{det}(e_{\lambda^t_i-i+j})_{1\leq i,j\leq \lambda_1}$$
letting $e_0=1$ and $e_i=0$ for $i<0$. 
More information can be found in Macdonald's book~\cite{M79} and Sagan's book~\cite{S01}.
\begin{example}\label{ex:s}
$s_{(2,1)}= e_{(2,1)}-e_{(3)}={x_1^2x_2+x_1x_2^2+\cdots} + {2x_1x_2x_3+\cdots} $
\end{example}

We say a function $F\in \Lambda$ is {\it $e$-positive}, respectively {\it Schur-positive}, if $F$ can be written as a non-negative sum of elementary symmetric functions, respectively Schur symmetric functions. We will often refer to a graph itself as $e$-positive or Schur-positive if its chromatic symmetric function is respectively $e$-positive or Schur-positive.
\begin{example}
The chromatic symmetric function for the complete graph is 
$$X_{K_n}=n!e_{(n)}=n!s_{(1^n)},$$
so is both $e$-positive and Schur-positive. 
\label{ex:complete_graph}
\end{example}
\begin{remark}
The paths~\cite[Proposition 5.3]{Stan95}, cycles~\cite[Proposition 5.4]{Stan95} and lollipops~\cite[Corollary 7.7]{GebSag} are well-known families of $e$-positive graphs. 
\label{re:e-pos_graphs}
\end{remark}

Stanley~\cite[Proposition 2.3]{Stan95} found that if $G\cup H$ is the disjoint union of two graphs $G$ and $H$ then 
$$X_{G\cup H}=X_GX_H.$$
\begin{remark}
Because the elementary basis is multiplicative if graphs $G$ and $H$ are $e$-positive then so is $G\cup H$. 
\label{re:disjoint}
\end{remark}

There are two other bases of symmetric functions that will be especially useful. These two bases are called the power-sum basis and the monomial basis. The {\it $i$-th power-sum symmetric function} is 
$$p_i=x_1^i+x_2^i+x_3^i+\cdots$$
and the {\it power-sum symmetric function} associated to $\lambda$ is 
$$p_{\lambda}=p_{\lambda_1}p_{\lambda_2}\cdots p_{\lambda_{l(\lambda)}}.$$
\begin{example}\label{ex:p}
$p_{(2,1)}=(x_1^2+x_2^2+x_3^2+\cdots)(x_1+x_2+x_3+\cdots )$
\end{example}
The {\it monomial symmetric function} associated to $\lambda$ is 
$$m_{\lambda}=\sum_{j_1,j_2,\svw{\ldots},j_{l(\lambda)}}x_{j_1}^{\lambda_1}x_{j_2}^{\lambda_2}\cdots x_{j_{l(\lambda)}}^{\lambda_{l(\lambda)}}$$
summed over distinct monomials. 
\begin{example}\label{ex:m}
$m_{(2,1)}=x_1^2x_2+x_1x_2^2+x_1^2x_3+x_1x_3^2+x_2^2x_3+x_2x_3^2+\cdots$
\end{example}

We will call a function {\it $m$-positive} if it can be written as a non-negative sum of monomial symmetric functions.

\begin{remark}
It is well known  that  all elementary   and Schur symmetric functions can be written as non-negative sums of monomial symmetric functions. Hence, if a symmetric function $F$ is $e$-positive or Schur-positive, then $F$ is $m$-positive. As result, if $F$ is not $m$-positive, then $F$ is not $e$-positive and not Schur-positive. For details see~\cite{M79, S01}.
\label{remark:mpos}
\end{remark}

Given any basis ${\mathcal B}=\{b_{\lambda}:\lambda\vdash n, n\geq 0\}$ of $\Lambda$ we define for $F\in \Lambda$ the notation $[b_{\lambda}]F$  to be the coefficient of $b_{\lambda}$ when $F$ is fully expanded in basis ${\mathcal B}$. At times we will even talk about coefficients of polynomials $p(k)$ in variable $k$, so let $[k^j]p(k)$ be the coefficient of  $k^j$ in $p(k)$. 

Stanley has several useful formulas for $X_G$ in terms of the power-sum and monomial bases. 
A {\it partition} of the vertices $V(G)$ is a collection of disjoint non-empty subsets of vertices, called {\it blocks}, whose full union is $V(G)$.  We say the partition is of {\it type} $\lambda\vdash\# V(G)$ if the relative sizes of the blocks form $\lambda$. A {\it stable partition} is a partition of the vertices so that each block is an {\it independent set} of $G$, meaning that $G$ restricted to each block has no edges. A {\it connected partition} is a partition of the vertices so that $G$ restricted to each block is connected. For an integer partition $\lambda=(1^{r_1},2^{r_2},\ldots, n^{r_n})$ define $\lambda^!=r_1!r_2!\cdots r_n!$. 
\begin{theorem}[Stanley~\cite{Stan95} Proposition 2.4]
For a graph $G$, 
$$X_G=\sum_{\lambda\vdash n}a_{\lambda}\lambda^!m_{\lambda}$$
where $a_{\lambda}$ is the number of stable partitions of $G$ of type $\lambda$. 
\label{thm:X_Gmonomial}
\end{theorem}
\begin{example}\label{ex:P3_in_m}
$X_{P_3}=m_{(2,1)}+6m_{(1,1,1)}$
\end{example}
\begin{theorem}[Stanley~\cite{Stan95} Theorem 2.5]For a graph $G$
$$X_G=\sum_{S\subseteq E(G)}(-1)^{\#S}p_{\lambda(S)} $$
where $\lambda(S)$ is the integer partition formed from the sizes of the connected components formed by restricting the graph $G$ to  $S$. 
\label{thm:X_Gpowersum}
\end{theorem}
\begin{example}\label{ex:P3_in_p}
$X_{P_3}=p_{(3)}-2p_{(2,1)}+p_{(1,1,1)}$ 
\end{example}

The goal of this paper is to investigate two posets on ${\mathcal G}_n$, equivalence classes of connected graphs on $n$ vertices decided by equivalent chromatic symmetric functions. The relations will reflect relative $e$-positivity or Schur-positivity of the chromatic symmetric functions of graphs. A {\it poset}, or partially ordered set, is a collection of objects and a relation $\leq$ between some of these objects that is reflexive, transitive and antisymmetric. For more information see~\cite{stanley1}. There are several elements or sets of elements that are studied in posets because of their particular properties. One group of elements are {\it maximal elements}, which are those $x\in P$ where $y\leq x$ for all comparable $y\in P$. Similarly, {\it minimal elements} are those $x\in P$ where $y\geq x$ for all comparable $y\in P$. A {\it chain}  in a poset is a collection of elements $Q\subseteq P$ such that all elements $x,y\in Q$ are related with $x\leq y$ or $y\leq x$. An {\it antichain} is a collection of elements $Q\subseteq P$ such that all distinct $x,y\in Q$ are unrelated with $x\not\leq y$ and $y\not\leq x$. We call $x\in P$ {\it independent} if $x$ is not related to any other element in $P$. An {\it interval} of a poset is $[x,y]=\{z\in P:x\leq z\leq y\}$. Every poset has an associated {\it M\"{o}bius function}, which is a map $\mu:P\times P\rightarrow \Z$ such that $\mu(x,x)=1$, {for $x<y$}
$$\sum_{z\in [x,y]}\mu(x,z)=0,$${and for incomparable $x$ and $y$ let $\mu(x,y)=0$.}

One natural way to define relative $e$-positivity or Schur-positivity of the chromatic symmetric functions is by considering the $e$-positivity or Schur-positivity  of the difference between two chromatic symmetric functions, $X_G-X_H$. However, we will find that $X_G-X_H$ is never $e$-positive or Schur-positive unless $X_G=X_H$. To show this we need the following lemma that will use  the following conversion formula going from the Schur basis to the power-sum basis~\cite[Theorem 4.6.4]{S01}, 
\begin{equation}
s_{\lambda}=\frac{1}{n!}\sum_{w\in\fS_n}\chi^{\lambda}(w)p_{w}
\label{eq:schur_powersum}
\end{equation}
where $\fS_n$ is the symmetric group,  $\chi^{\lambda}(w)$ is the irreducible character of $\lambda$ evaluated at $w$ and if $w$ is of cycle type $\lambda\vdash n$ then $p_{w}=p_{\lambda}$. 

\begin{lemma} If $F\in\Lambda^n$ is a non-zero Schur-positive function  then $[p_{(1^n)}]F>0$.
\label{lem:conversion_note}
\end{lemma}
\begin{proof}
Note that in the conversion formula from Schur basis to the power-sum basis in equation~\eqref{eq:schur_powersum}  the only 
 time $p_{(1^n)}$ appears is for the identity permutation, $\pi=\text{id}$. It is a fact that for any irreducible 
character $\chi^{\lambda}(\text{id})>0$ because $\chi^{\lambda}(\text{id})$ is the dimension of the character~\cite[Proposition 1.8.5]{S01}. This means that $[p_{(1^n)}]s_{\lambda}>0$ for all $\lambda\vdash n$. Further if we are given a non-zero Schur-positive function $F\in\Lambda^n$ then $[p_{(1^n)}]F>0$. It follows that if $F\in\Lambda^n$ is Schur-positive and $[p_{(1^n)}]F=0$ then $F=0$. 
\end{proof}

 \begin{theorem}
For all graphs $G$ and $H$ on $n$ vertices we have either $X_G-X_H=0$ or $X_G-X_H$ is not Schur-positive.
\label{thm:Schur_X_G-X_H}
 \end{theorem}  
 \begin{proof}
Using the formula in Theorem~\ref{thm:X_Gpowersum} we can see that for a  graph $G$ the only way to get a term $p_{(1^n)}$ is to disconnect every vertex. This is only possible when using $S=\emptyset$, the empty edge subset. This shows that $[p_{(1^n)}]X_G=1$ for all graphs $G$. Further this implies that for two graphs $G$ and $H$ on $n$ vertices that $[p_{(1^n)}](X_G-X_H)=0$. By Lemma~\ref{lem:conversion_note} we can then see that if $F\in\Lambda^n$ is Schur-positive and $[p_{(1^n)}]F=0$ then $F=0$.  So, if $X_G-X_H$ is a Schur-positive function then we can conclude that $X_G-X_H=0$. Thus,  $X_G-X_H$ is either not Schur-positive or $X_G-X_H=0$. 
\end{proof}

We can similarly get the same result for $e$-positivity. 

 \begin{theorem}
For all graphs $G$ and $H$ on $n$ vertices we have either $X_G-X_H$ is 0 or not $e$-positive.
\label{thm:elem_X_G-X_H}
 \end{theorem}  
\begin{proof}
This follows  from Theorem~\ref{thm:Schur_X_G-X_H}: If $X_G-X_H$ is $e$-positive, then since all $e$-positive functions are Schur-positive,
 we then know that $X_G-X_H$ is Schur-positive, which implies that $X_G-X_H=0$. 
\end{proof}

Because direct subtraction between two distinct chromatic symmetric functions is never $e$-positive or Schur-positive, defining the relation between $G$ and $H$ based on direct subtraction  gives a poset with no relations. Hence, we base our relations in our posets on the following  weighted subtractions. The goal of these subtractions is to zero-out the $e_{(n)}$ or $s_{(1^n)}$ term, which is never zero in a chromatic symmetric function. This is a  fact that will be evident later from theorems presented further on in this section. 
Define 
$$X_e(G,H)=X_G-\frac{[e_{(n)}]X_G}{[e_{(n)}]X_H}X_H$$
and
$$X_s(G,H)=X_G-\frac{[s_{(1^n)}]X_G}{[s_{(1^n)}]X_H}X_H.$$
Define  $\cE_n$ to be the poset on ${\mathcal G}_n$ related to the elementary basis. We will say $G\geq_e H$ if and only if $X_e(G,H)$ is $e$-positive. Similarly,  $\cS_n$ will be the poset on ${\mathcal G}_n$ related to the Schur basis. We will say $G\geq_s H$ if and only if $X_s(G,H)$ is Schur-positive. See Figure~\ref{fig:posets} for examples. 

Because our relations depend on weighted subtractions determined by the coefficients of $e_{(n)}$ and $s_{(1^n)}$ we will need some machinery to determine these coefficients. The following propositions gives a way to calculate these coefficients via their chromatic polynomial and an interpretation in terms of acyclic orientations. 
 An {\it orientation} of a graph $G$ is an  assignment for each edge between $u$ and $v$ a direction from  $u$ to $v$ or from $v$ to $u$. We call an orientation {\it acyclic} if there are no directed cycles. Given an orientation on $G$ we call a vertex $v$ a {\it sink} if all edges incident to $v$ are directed towards $v$ and a {\it source} if all edges incident to $v$ are directed away from $v$. Let $S(G,j)$ count the number of acyclic orientations of $G$ with exactly $j$ sinks. 
 Stanley~\cite[Theorem 3.3]{Stan95} found that for any connected graph $G$ on $n$ vertices with $[e_{\lambda}]X_G=c_{\lambda}$ that
\begin{equation}
\underset{ l(\lambda)=j}{\sum_{\lambda\vdash n}}c_{\lambda}=S(G,j),
\label{eq:elem_acyclic}
\end{equation}
which is a refinement of a result in~\cite{Stan73} that the total number of acyclic orientations is
\begin{equation}
\sum_{j=0}^nS(G,j)=(-1)^{n}\chi_G(-1).
\label{eq:total_acyclic}
\end{equation}

\begin{theorem}[Greene and Zaslavsky~\cite{GZ83} Theorem 7.3]
Given a graph $G$ on $n$ vertices and $v\in V(G)$ the number of acyclic orientations with $v$ as a unique sink is $(-1)^{n-1} [k]\chi_G(k)$. Also, 
$$S(G,1)=(-1)^{n-1}n\cdot [k]\chi_G(k)$$
and unless $G$ has no edges $S(G,1)>0$.
\label{thm:S(G,1)}
\end{theorem}

Putting together equation~\eqref{eq:elem_acyclic} and Theorem~\ref{thm:S(G,1)} we have
\begin{equation}
[e_{(n)}]X_G=(-1)^{n-1}n\cdot [k]\chi_G(k)=S(G,1).
\label{eq:e_n_coeff}
\end{equation}

\begin{theorem}[Kaliszewski~\cite{Kaliszwecki} Theorem 1.1]
The coefficient of the hook shape $(k,1^{n-k})$ in the Schur basis is
$$[s_{(k,1^{n-k})}]X_G=\sum_{j=1}^n\binom{j-1}{k-1}S(G,j)$$
for $k\in[n]$. 
Specifically $[s_{(1^{n})}]X_G$ counts the total number of acyclic orientations of $G$.
\label{thm:s_hook_coeff}
\end{theorem}
Using equation~\eqref{eq:total_acyclic} we can see from Theorem~\ref{thm:s_hook_coeff} that   
\begin{equation}
[s_{(1^n)}]X_G=(-1)^{n}\chi_G(-1).
\label{eq:s_1^n_coeff}
\end{equation}

Using these methods on trees we get the following formulas for  two coefficients. 

\begin{corollary}
If $T$ is a tree on $n$ vertices then $[e_{(n)}]X_T=n$ and $[s_{(1^n)}]X_T=2^{n-1}$. 
\label{cor:X_T_coeff}
\end{corollary}
\begin{proof}
It is known that the chromatic polynomial for a tree on $n$ vertices is $\chi_T(k)=k(k-1)^{n-1}$. 
From equation~\eqref{eq:e_n_coeff} we can see that $[e_{(n)}]X_T=n$ and from equation~\eqref{eq:s_1^n_coeff} we can see that $[s_{(1^n)}]X_T=2^{n-1}$. 
\end{proof}
Also
using these methods {we} get the following formulas for the two coefficients for complete graph{s}. 

\begin{proposition}We have 
$[e_{(n)}]X_{K_n}=[s_{(1^n)}]X_{K_n}=n!$.
\label{prop:X_K_n} 
\end{proposition}
\begin{proof}
This follows immediately from Example~\ref{ex:complete_graph}.
\end{proof}

One motivating property for our posets $\cE_n$ and $\cS_n$ is that they give a new equivalent condition for $e$-positivity and Schur-positivity.

\begin{theorem}
We have $G\geq_e K_n$ if and only if $G$ is $e$-positive.
\label{thm:G>K_n}
\end{theorem}
\begin{proof}
Note that because of Proposition~\ref{prop:X_K_n}
$$X_e(G,K_n)=X_G-\frac{[e_{(n)}]X_G}{[e_{(n)}]X_{K_n}}X_{K_n}=X_G-[e_{(n)}]X_G\cdot e_{(n)},$$
which is $X_G$ without its $e_{(n)}$ term. We can see from Theorem~\ref{thm:S(G,1)} and equation~\eqref{eq:e_n_coeff} that  $[e_{(n)}]X_G>0$. Because of this $X_G$ is not $e$-positive if there is a $\lambda\vdash n$ with $\lambda\neq (n)$ such that $[e_{\lambda}]X_G<0$. Altogether $X_G$ is $e$-positive if and only if $X_e(G,K_n)=X_G-[e_{(n)}]X_G\cdot e_{(n)}$ is $e$-positive. 
\end{proof}
 
 We have a similar condition for our other poset. 
 
 \begin{theorem}
We have $G\geq_s K_n$ if and only if $G$ is Schur-positive.
\label{thm:G>K_n_schur}
\end{theorem}
\begin{proof}This is similar to the proof of Theorem~\ref{thm:G>K_n}. 
Note that because of Proposition~\ref{prop:X_K_n}
$$X_s(G,K_n)=X_G-\frac{[s_{(1^n)}]X_G}{[s_{(1^n)}]X_{K_n}}X_{K_n}=X_G-[s_{(1^n)}]X_G\cdot s_{(1^n)},$$
which is $X_G$ without its $s_{(1^n)}$ term. We can see from Theorem~\ref{thm:s_hook_coeff}  and Theorem~\ref{thm:S(G,1)} that $[s_{(1^n)}]X_G>0$. Because of this $X_G$ is not Schur-positive if there is a $\lambda\vdash n$ with $\lambda\neq (1^n)$ such that $[s_{\lambda}]X_G<0$. Altogether $X_G$ is Schur-positive if and only if $X_s(G,K_n)=X_G-[s_{(1^n)}]X_G\cdot s_{(1^n)}$ is Schur-positive. 
\end{proof}
 
Normally $e$-positivity implies Schur-positivity, but in general a relation in either poset  does not imply a relation in the other.

\begin{figure}
\begin{tikzpicture}
\draw[black] (2.5,1.5)--(1.5,2.5);
\draw[black] (2.5,1.5)--(3,2);
\draw[black] (4,3.1)--(4,4.2);
\draw[black] (2.5,5.5)--(1.5,4.5);
\draw[black] (2.5,5.5)--(3.25,4.75);
\begin{scope}[shift={(-.25,1.5)}]
\coordinate (A) at (0,0);
\coordinate (B) at (1,0);
\coordinate (C) at (-1,-.5);
\coordinate (D) at (-1,.5);
\draw[black] (A)--(B);
\draw[black] (C)--(A)--(D);
\filldraw[black] (A) circle [radius=2pt];
\filldraw[black] (B) circle [radius=2pt];
\filldraw[black] (C) circle [radius=2pt];
\filldraw[black] (D) circle [radius=2pt];
\end{scope}
\begin{scope}[shift={(2,0)}]
\coordinate (A) at (0,1);
\coordinate (B) at (0,0);
\coordinate (C) at (1,0);
\coordinate (D) at (1,1);
\draw[black] (A)--(B)--(C)--(D)--(A)--(C);
\draw[black] (D)--(B);
\filldraw[black] (A) circle [radius=2pt];
\filldraw[black] (B) circle [radius=2pt];
\filldraw[black] (C) circle [radius=2pt];
\filldraw[black] (D) circle [radius=2pt];
\end{scope}
\begin{scope}[shift={(3,1.75)}]
\coordinate (A) at (0,.5);
\coordinate (B) at (1,0);
\coordinate (D) at (1,1);
\coordinate (C) at (2,.5);
\draw[black] (A)--(B)--(C)--(D)--(A);
\draw[black] (B)--(D);
\filldraw[black] (A) circle [radius=2pt];
\filldraw[black] (B) circle [radius=2pt];
\filldraw[black] (C) circle [radius=2pt];
\filldraw[black] (D) circle [radius=2pt];
\end{scope}
\begin{scope}[shift={(1,3)}]
\coordinate (A) at (0,0);
\coordinate (B) at (1,0);
\coordinate (C) at (1,1);
\coordinate (D) at (0,1);
\draw[black] (A)--(B)--(C)--(D)--(A);
\filldraw[black] (A) circle [radius=2pt];
\filldraw[black] (B) circle [radius=2pt];
\filldraw[black] (C) circle [radius=2pt];
\filldraw[black] (D) circle [radius=2pt];
\end{scope}
\begin{scope}[shift={(3,4.5)}]
\coordinate (A) at (2,.5);
\coordinate (B) at (2,-.5);
\coordinate (C) at (0,0);
\coordinate (D) at (1,0);
\draw[black] (C)--(D)--(A)--(B);
\draw[black] (D)--(B);
\filldraw[black] (A) circle [radius=2pt];
\filldraw[black] (B) circle [radius=2pt];
\filldraw[black] (C) circle [radius=2pt];
\filldraw[black] (D) circle [radius=2pt];
\end{scope}
\begin{scope}[shift={(1,6)}]
\coordinate (A) at (0,0);
\coordinate (B) at (1,0);
\coordinate (C) at (2,0);
\coordinate (D) at (3,0);
\draw[black] (A)--(B)--(C)--(D);
\filldraw[black] (A) circle [radius=2pt];
\filldraw[black] (B) circle [radius=2pt];
\filldraw[black] (C) circle [radius=2pt];
\filldraw[black] (D) circle [radius=2pt];
\end{scope}
\end{tikzpicture}
\hspace{1in}
\begin{tikzpicture}
\draw[black] (2.5,1.1)--(2.5,1.6);
\draw[black] (2.5,3)--(1.75,3.75);
\draw[black] (2.5,3)--(3.5,4);
\draw[black] (2.5,5.75)--(2,5.25);
\draw[black] (2.5,5.75)--(3.25,4.75);
\begin{scope}[shift={(-.25,1.5)}]
\coordinate (A) at (0,0);
\coordinate (B) at (1,0);
\coordinate (C) at (-1,-.5);
\coordinate (D) at (-1,.5);
\draw[black] (A)--(B);
\draw[black] (C)--(A)--(D);
\filldraw[black] (A) circle [radius=2pt];
\filldraw[black] (B) circle [radius=2pt];
\filldraw[black] (C) circle [radius=2pt];
\filldraw[black] (D) circle [radius=2pt];
\end{scope}
\begin{scope}[shift={(2,0)}]
\coordinate (A) at (0,1);
\coordinate (B) at (0,0);
\coordinate (C) at (1,0);
\coordinate (D) at (1,1);
\draw[black] (A)--(B)--(C)--(D)--(A)--(C);
\draw[black] (D)--(B);
\filldraw[black] (A) circle [radius=2pt];
\filldraw[black] (B) circle [radius=2pt];
\filldraw[black] (C) circle [radius=2pt];
\filldraw[black] (D) circle [radius=2pt];
\end{scope}
\begin{scope}[shift={(1.5,1.75)}]
\coordinate (A) at (0,.5);
\coordinate (B) at (1,0);
\coordinate (D) at (1,1);
\coordinate (C) at (2,.5);
\draw[black] (A)--(B)--(C)--(D)--(A);
\draw[black] (B)--(D);
\filldraw[black] (A) circle [radius=2pt];
\filldraw[black] (B) circle [radius=2pt];
\filldraw[black] (C) circle [radius=2pt];
\filldraw[black] (D) circle [radius=2pt];
\end{scope}
\begin{scope}[shift={(1,4)}]
\coordinate (A) at (0,0);
\coordinate (B) at (1,0);
\coordinate (C) at (1,1);
\coordinate (D) at (0,1);
\draw[black] (A)--(B)--(C)--(D)--(A);
\filldraw[black] (A) circle [radius=2pt];
\filldraw[black] (B) circle [radius=2pt];
\filldraw[black] (C) circle [radius=2pt];
\filldraw[black] (D) circle [radius=2pt];
\end{scope}
\begin{scope}[shift={(3,4.5)}]
\coordinate (A) at (2,.5);
\coordinate (B) at (2,-.5);
\coordinate (C) at (0,0);
\coordinate (D) at (1,0);
\draw[black] (C)--(D)--(A)--(B);
\draw[black] (D)--(B);
\filldraw[black] (A) circle [radius=2pt];
\filldraw[black] (B) circle [radius=2pt];
\filldraw[black] (C) circle [radius=2pt];
\filldraw[black] (D) circle [radius=2pt];
\end{scope}
\begin{scope}[shift={(1,6)}]
\coordinate (A) at (0,0);
\coordinate (B) at (1,0);
\coordinate (C) at (2,0);
\coordinate (D) at (3,0);
\draw[black] (A)--(B)--(C)--(D);
\filldraw[black] (A) circle [radius=2pt];
\filldraw[black] (B) circle [radius=2pt];
\filldraw[black] (C) circle [radius=2pt];
\filldraw[black] (D) circle [radius=2pt];
\end{scope}
\end{tikzpicture}
\caption{On the left we have $\cE_4$ and on the right we have $\cS_4$.}
\label{fig:posets}
\end{figure}

 \begin{proposition}\label{prop:eisnots}
The relation $G<_eH$ does not necessarily imply $G<_sH$.   The relation $G<_sH$ does not necessarily imply $G<_eH$. 
 \end{proposition}
\begin{proof}
See Figure~\ref{fig:graph_ex2} for the pictures of the graph in this proof. We can see that the diamond $D$ is less than $C_4$ in $\cS_4$, however $D$ is not less than $C_4$ in $\cE_4$, which can be seen in Figure~\ref{fig:posets}. We also find that for $G_1$ and $G_2$ given in Figure~\ref{fig:graph_ex2} that $G_1>_e G_2$ in $\cE_7$ but $G_1\not>_sG_2$ in $\cS_7$, looking at the coefficient of $s_{(2^3,1)}$. 
\end{proof}

\begin{figure}[h]
\begin{tikzpicture}
\begin{scope}[shift={(0,0)}]
\coordinate (A) at (0,.5);
\coordinate (B) at (1,0);
\coordinate (D) at (1,1);
\coordinate (C) at (2,.5);
\draw[black] (A)--(B)--(C)--(D)--(A);
\draw[black] (B)--(D);
\filldraw[black] (A) circle [radius=2pt];
\filldraw[black] (B) circle [radius=2pt];
\filldraw[black] (C) circle [radius=2pt];
\filldraw[black] (D) circle [radius=2pt];
\end{scope}
\begin{scope}[shift={(3,0)}]
\coordinate (A) at (0,0);
\coordinate (B) at (1,0);
\coordinate (C) at (1,1);
\coordinate (D) at (0,1);
\draw[black] (A)--(B)--(C)--(D)--(A);
\filldraw[black] (A) circle [radius=2pt];
\filldraw[black] (B) circle [radius=2pt];
\filldraw[black] (C) circle [radius=2pt];
\filldraw[black] (D) circle [radius=2pt];
\end{scope}
\begin{scope}[shift={(5,0)}]
\coordinate (A) at (0,0);
\coordinate (B) at (1,0);
\coordinate (C) at (1,1);
\coordinate (D) at (0,1);
\coordinate (E) at (2,.5);
\coordinate (F) at (3,.5);
\coordinate (G) at (4,.5);
\draw[black] (A)--(B)--(C)--(D)--(A);
\draw[black] (C)--(E)--(F)--(G);
\filldraw[black] (A) circle [radius=2pt];
\filldraw[black] (B) circle [radius=2pt];
\filldraw[black] (C) circle [radius=2pt];
\filldraw[black] (D) circle [radius=2pt];
\filldraw[black] (E) circle [radius=2pt];
\filldraw[black] (F) circle [radius=2pt];
\filldraw[black] (G) circle [radius=2pt];
\end{scope}
\begin{scope}[shift={(10.5,0)}]
\coordinate (A) at (0,0);
\coordinate (B) at (1,0);
\coordinate (C) at (2,0);
\coordinate (D) at (2.5,1);
\coordinate (E) at (1.5,1);
\coordinate (F) at (.5,1);
\coordinate (G) at (-.5,1);
\draw[black] (A)--(B)--(C)--(D)--(E)--(F)--(G)--(A);
\filldraw[black] (A) circle [radius=2pt];
\filldraw[black] (B) circle [radius=2pt];
\filldraw[black] (C) circle [radius=2pt];
\filldraw[black] (D) circle [radius=2pt];
\filldraw[black] (E) circle [radius=2pt];
\filldraw[black] (F) circle [radius=2pt];
\filldraw[black] (G) circle [radius=2pt];
\end{scope}
\end{tikzpicture}

\begin{align*}
X_D&=16e_{(4)}+2e_{(3,1)}=18s_{(1^4)}+2s_{(2,1,1)}\\
X_{C_4}&=12e_{(4)}+2e_{(2,2)}=14s_{(1^4)}+2s_{(2,1,1)}+2s_{(2,2)}\\
X_{G_1}&= 21e_{(7)}+15e_{(6,1)}+29e_{(5,2)}+27e_{(4,3)}+12e_{(4,2,1)}+7e_{(3,2^2)}+e_{(2^3,1)}\\
&=112s_{(1^7)} + 112s_{(2,1^5)}+ 106s_{(2^2,1^3)}+ 57s_{(2^3,1)} + 22s_{(3,14)}  + 10s_{(3,2^2)} \\
&\hspace{.2in}+ 32s_{(3,2,1^2)}+ 10s_{(3^2,1)} + s_{(4,1^3)}+ 2s_{(4,3,1)} + s_{(4,3)}\\
X_{G_2}&= 42e_{(7)}+28e_{(5,2)}+42e_{(4,3)}+14e_{(3,2^2)}\\
&=126s_{(1^7)} + 98s_{(2,1^5)}+ 112s_{(2^2,1^3)}+ 70s_{(2^3,1)} + 14s_{(3,1^4)} + 28s_{(3,2,1^2)}\\
&\hspace{.2in}+  14s_{(3,2^2)} +14s_{(3^2,1)}\\
\end{align*}
\vspace{-.4in}

\caption{From left to right we have the diamond $D$, $C_4$,  $G_1$ and $G_2$. }
\label{fig:graph_ex2}
\end{figure}

However, this is true for trees. Furthermore, we will find later that trees form an anti-chain in both families of posets. 

 \begin{proposition}
For two trees if $T_1>_eT_2$ then $T_1>_sT_2$. 
 \end{proposition}
 
\begin{proof}
Let $T_1$ and $T_2$ be two trees on $n$ vertices. By Corollary~\ref{cor:X_T_coeff} we have $X_e(T_1,T_2)=X_{T_1}-X_{T_2}$ and $X_s(T_1,T_2)=X_{T_1}-X_{T_2}$, which are equal.  Hence, if $X_{T_1}-X_{T_2}$ is $e$-positive, $T_1>_eT_2$, then $X_{T_1}-X_{T_2}$ is Schur-positive, $T_1>_sT_2$. 
\end{proof}

\section{Properties of the chromatic $e$-positivity poset}
\label{sec:elem}

In this section we show several  properties of $\cE_n$. In Theorem~\ref{thm:G>K_n} we proved that $G\geq_e K_n$ if and only if $G$ is $e$-positive, and now we will prove that $K_n$ is  a minimal element. We will find that  as we increase in $\cE_n$ the  independence number, $\alpha(G)$, increases, the number of acyclic orientations with one sink, $S(G,1)$, decreases, and the chromatic number, $\chi(G)$, decreases. Also, we will show that trees with distinct chromatic symmetric functions form an anti-chain and are maximal elements. In particular, the stars $S_n$ are independent elements, so $\cE_n$ can not be a lattice. Lastly, we show that collection of lollipops form a chain with the complete graph as the minimal element and the path as the maximal element. 

Though it will be shown later in Corollary~\ref{cor:K_n_min_elem} that the complete graph is a minimal element and in Proposition~\ref{prop:trees_maximal_elem} that trees are maximal elements in $\cE_n$, we  consider several common graph statistics  that have the complete graph and  trees at the two extremal bounds of the statistic. We determine  if the statistic increases or decreases with relations in the poset. The statistics we consider here are the  the independence number, $\alpha(G)$, the number of acyclic orientations with one sink, $S(G,1)$, the chromatic number $\chi(G)$, and the clique number, $\omega(G)$.

\begin{proposition}
If $G\geq_e H$ then $\alpha(G)\geq \alpha(H)$. 
\label{prop:alpha}
\end{proposition}
\begin{proof}
Suppose $\alpha(G) < \alpha(H)$ and let $\lambda$ be the partition of $n$ given by $(\alpha(H), 1^{n-\alpha(H)}).$ By definition, $H$ has a  stable partition of type $\lambda$ but $G$ does not. Hence, by the expansion of $X_G$ and $X_H$ into the monomial basis, Theorem~\ref{thm:X_Gmonomial}, we have $[m_\lambda] X_G = 0$ but $[m_\lambda] X_H \geq 1$ since the coefficient of $m_\lambda$ in $X_G$  is a multiple of the number of stable partitions of type $\lambda$ in $G$. Therefore, $[m_\lambda] X_e(G,H) < 0$  since equation~\eqref{eq:e_n_coeff} implies that the scaling factor is always positive. Hence  by Remark~\ref{remark:mpos}, $X_e(G,H)$ cannot be $e$-positive as $X_e(G,H)$ is not $m$-positive.
\end{proof}

\begin{proposition}
If $G>_e H$ then $S(G,1)<S(H,1)$. 
\label{prop:S(G,1)<S(H,1)}
\end{proposition}
\begin{proof}
Let $G>_e H$ so $X_e(G,H)\neq 0$ is $e$-positive. It is well-known that $e$-positive functions are Schur-positive, so $X_e(G,H)\neq 0$ is Schur-positive. By Theorem~\ref{thm:X_Gpowersum} we can see that $[p_{(1^n)}]X_G=[p_{(1^n)}]X_H=1$ so we have that
$$[p_{(1^n)}]X_e(G,H)=1-\frac{[e_{(n)}]X_G}{[e_{(n)}]X_H}.$$
By Lemma~\ref{lem:conversion_note} this coefficient is positive so $[e_{(n)}]X_G<[e_{(n)}]X_H$. By equation~\eqref{eq:total_acyclic} we have our result. 
\end{proof}

\begin{proposition}
If $G \geq_e H$, then $\chi(G) \leq \chi(H)$.
\end{proposition}

\begin{proof}
Suppose $\chi(G) > \chi(H)$. Since there is a coloring of $H$ with $\chi(H)$ colors, then $H$ has a stable partition of some type $\lambda \vdash n$ with length $\chi(H)$ 
and hence, we have $[m_\lambda] X_H \geq 1$ by Theorem~\ref{thm:X_Gmonomial}. Furthermore, $[m_\lambda] X_G = 0$ since by definition, $G$ cannot be colored with fewer than $\chi(G)$ colors and by assumption $\chi(G) > \chi(H)$. Therefore, $[m_\lambda] X_e(G,H) < 0$  since equation~\eqref{eq:e_n_coeff} implies that the scaling factor is always positive. Hence by Remark~\ref{remark:mpos}, $X_e(G,H)$ cannot be $e$-positive  as  $X_e(G,H)$ is  not $m$-positive.
\end{proof}

\begin{remark}
Note that there is no consistent relationship between  clique numbers $\omega(G)$ and relations in the poset $\cE_n$. We find that $P_5>_e L_{3,2}>_e C_5$ but $\omega(P_5)<\omega(L_{3,2})>\omega(C_5)$. 
\end{remark}

Earlier in Theorem~\ref{thm:G>K_n} we proved that $G$ $e$-positive if and only if $G\geq_e K_n$. We can further show that $K_n$ is a minimal element in our poset $\cE_n$. 

\begin{corollary}
The complete graph $K_n$ is a minimal element in ${\mathcal E}_n$.
\label{cor:K_n_min_elem}
\end{corollary}
\begin{proof}
If $G\neq K_n$ is a connected graph on $n$ vertices then there are two vertices $u$ and $v$ without an edge between them. The set $\{u,v\}$ is an independent set and  so $\alpha(G)\geq 2$. By Proposition~\ref{prop:alpha} we know $K_n \not\geq_e G$ because $1= \alpha(K_n)<\alpha(G)$. Hence $K_n$ is a minimal element. 
\end{proof}

The next proposition gives us an  elegant condition that can generate independence sets in $\cE_n$.

\begin{theorem}
Let $\{G_1, G_2, \ldots, G_k\}$ be some set of connected graphs on $n$ vertices with equal $e_{(n)}$ coefficients, $[e_{(n)}]X_{G_i}=[e_{(n)}]X_{G_j}$, and distinct chromatic symmetric functions. Then 
$\{G_1, G_2, \ldots, G_k\}$ is an anti-chain  in  ${\mathcal E}_n$. 
\label{thm:elem_antichain}
\end{theorem}
\begin{proof}
Let $\{G_1, G_2, \ldots, G_k\}$ satisfy all the assumptions stated in the proposition. This means $X_e(G_i,G_j)=X_{G_i}-X_{G_j}$. By Theorem~\ref{thm:elem_X_G-X_H} we know that either $X_{G_i}-X_{G_j}=0$ so $i=j$ or $X_{G_i}-X_{G_j}$ is not $e$-positive. This implies that $G_i$ is not related to $G_j$ for all $i\neq j$, which means that $\{G_1, G_2, \ldots, G_k\}$ is an anti-chain  in  ${\mathcal E}_n$. 
\end{proof}

\begin{remark}
By Theorem~\ref{thm:elem_antichain} given any integer $z\in\Z$ the collection of graphs on $n$ vertices $\{G:[e_{(n)}]X_G=z\}$ is an anti-chain in $\cE_n$ under the assumption that we are grouping graphs together in $\cE_n$ if they have equal chromatic symmetric function. 
\end{remark}

\begin{corollary}
Trees on $n$ vertices with distinct chromatic symmetric functions form an anti-chain  in  ${\mathcal E}_n$.
\label{cor:trees_max_elem}
\end{corollary}
\begin{proof}
By Corollary~\ref{cor:X_T_coeff} all trees $T$ on $n$ vertices have $[e_{(n)}]X_T=n$, so by Theorem~\ref{thm:elem_antichain} we have our result. 
\end{proof}

Next we will show that trees not only form an anti-chain, but that they  are actually maximal elements in  ${\mathcal E}_n$.  In order to do this we will need a lower bound on $[e_{(n)}]X_G$.

\begin{lemma}
For any connected graph $G$ on $n$ vertices with  $\epsilon\in E(G)$
\begin{enumerate}[(i)]
\item $[e_{(n)}]X_G=[e_{(n)}]X_{G-\epsilon}+\frac{n}{n-1}X_{G/\epsilon}$, 
\item $[e_{(n)}]X_G\geq n(\# E(G)-n+2)$ and 
\item $[e_{(n)}]X_G\geq n$.
\end{enumerate}
\label{lem:edge_elem_bound}
\end{lemma}
\begin{proof}We can use the chromatic polynomial to calculate $[e_{(n)}]X_G$ via equation~\eqref{eq:e_n_coeff}. 
Using deletion-contraction from equation~\eqref{eq:deletion-contraction} on the chromatic polynomial in this formula gives us
\begin{align*}
[e_{(n)}]X_G&=(-1)^{n-1}n\cdot [k](\chi_{G-\epsilon}(k)-\chi_{G/\epsilon}(k))\\
&=(-1)^{n-1}n\cdot [k]\chi_{G-\epsilon}(k)+\frac{n}{n-1}(-1)^{n-2}(n-1)\cdot [k]X_{G/\epsilon}\\
&=[e_{(n)}]X_{G-\epsilon}+\frac{n}{n-1}X_{G/\epsilon}.
\end{align*}

To prove (ii) we will induct on the number of edges. If $G$ has the minimal number of edges for a connected graph then $G$ is a tree and $[e_{(n)}]X_G=n$ by Corollary~\ref{cor:X_T_coeff}, which satisfies the inequality. Let $\# E(G)>n-1$. This means that $G$ is not a tree and there exists a cycle. Let $\epsilon$ be an edge on one of these cycles, so $G-\epsilon$ is connected. Using part (i), induction and the fact that $\#E(G/\epsilon)\geq n-2$ we have 
\begin{align*}
[e_{(n)}]X_G&=[e_{(n)}]X_{G-\epsilon}+\frac{n}{n-1}X_{G/\epsilon}\\
&\geq n(\#E(G-\epsilon)-n+2)+n(\#E(G/\epsilon)-n+3)\\
&\geq n(\# E(G)-n+2).
\end{align*}
Because $G$ is a connected graph, so $\# E(G)\geq n-1$, we can conclude from part (ii) that $[e_{(n)}]X_G\geq n$.
\end{proof}

\begin{proposition}
All trees on $n$ vertices are maximal elements in  ${\mathcal E}_n$. 
\label{prop:trees_maximal_elem}
\end{proposition}
\begin{proof}
Suppose that $G >_e T$ for some tree $T$ and a connected graph $G$ each on $n$ vertices. By assumption, $X_e(G, T)$ is non-zero and $e$-positive, using Proposition~\ref{prop:S(G,1)<S(H,1)} and equation~\eqref{eq:e_n_coeff}, we have $[e_{(n)}] X_G < [e_{(n)}] X_T.$ Hence, $[e_{(n)}] X_G < n$ by Corollary~\ref{cor:X_T_coeff} since $T$ was a tree. However, Lemma~\ref{lem:edge_elem_bound} (iii) implies that $[e_{(n)}] X_G \geq n$ for any connected graph $G$, so we have reached a contradiction. Hence, $T$ is maximal.   
\end{proof}

While all trees are maximal elements there  are many that are specifically independent elements. One such family of trees are the star graphs $S_n$. To show this we will need several lemmas and {well-known results}. One {well-known result} we need for our lemma is Stanley's formula for $X_G$ in the power-sum basis that is summed over the bond lattice of the graph. The {\it bond lattice}, $L_G$, of a graph is a poset with vertices formed from connected partitions of the vertices $V(G)$. Let $\pi_1$ and $\pi_2$ be two connected partitions of $V(G)$. We say $\pi_1\leq_L\pi_2$ if all blocks in $\pi_1$ are a subset of some block of $\pi_2$. Let $\mu_L$ be the M\"{o}bius function of the poset $L_G$. 
Stanley's formula~\cite[Theorem 2.6]{Stan95} is 
\begin{equation}
X_G=\sum_{\pi\in L_G}\mu_L(\hat 0,\pi)p_{\pi}
\label{eq:X_Gpowersum_poset}
\end{equation}
where $\hat 0$ is the partition with all vertices in their own block and $p_{\pi}$ is the power-sum function associated to the sizes of the blocks in $\pi$. It is known, for example \cite{Stan95}, that $(-1)^{n-1}\mu_L(\hat 0,\hat 1)>0$ where $\hat 1$ is the partition with one block and that $(-1)^{n- l(\pi)}\mu_L(\hat 0,\pi)>0$. From this we know $(-1)^{n-1}[p_{(n)}]X_G>0$ and the sign of any coefficient of $p_{\lambda}$ in $X_G$. 

Another result we will need is the  the Newtonian identity~\cite[Proposition 7.7.6]{stanley2} 
\begin{equation}
p_n=\sum_{\lambda=(1^{m_1},2^{m_2},\ldots, n^{m_n})}\frac{(-1)^{n- l(\lambda)}n( l(\lambda)-1)!}{m_1!m_2!\cdots m_n!}e_{\lambda}.
\label{eq:powersum_to_elem}
\end{equation}

\begin{lemma}
Let $G$ be a connected graph with $\#V(G)\geq 4$ that is not a tree. Then 
\begin{enumerate}[(i)]
\item $G$ has a connected partition of type $\lambda=(\lambda_1,\lambda_2)$ with all parts  more than one and 
\item $[e_{\lambda}]X_G+[e_{(n)}]X_G>0$ if $\lambda_1\neq \lambda_2$ or
\item $[e_{\lambda}]X_G+\frac{[e_{(n)}]X_G}{2}>0$ if $\lambda_1= \lambda_2$. 
\end{enumerate}
\label{lem:nontrees}
\end{lemma}
\begin{proof}
First we will show  part (i) by inducting on the number of edges. The minimal number of edges for connected non-tree graphs $G$  is $n$ where $G$ is a unicyclic graph, meaning it has only one cycle. Note that if we remove any two edges from the cycle we have disconnected the graph into two connected components, and that the sizes of these components form a connected partition. However, we are not guaranteed that the sizes are more than one. First consider the case where our unique cycle has three vertices. Because $\#V(G)\geq 4$ there is some vertex $v$ in the cycle connected to another vertex $u$ outside the cycle. If we remove the two edges in the cycle incident to $v$ then we will have two components of size more than one. Next consider the case where the unique cycle has four or more vertices. Removing any two edges that do not share a vertex will guarantee that the two connected components have size more than one. This completes the base case. 

Assume that $\#E(G)>n$ and that any connected non-tree graph $H$ with $\#E(H)<\#E(G)$ has a connected partition of type $\lambda=(\lambda_1,\lambda_2)$ where all parts are more than one. There will exist some edge $\epsilon\in E(G)$ whose removal does not disconnect $G$. Because $\#E(G-\epsilon)\geq n$ we know that $G-\epsilon$ is a connected non-tree graph. By induction $G-\epsilon$ has a connected partition $\lambda=(\lambda_1,\lambda_2)$ with the required conditions, so we can conclude that $G$ does as well. This completes our proof of part (i). 

Let   $\lambda=(\lambda_1,\lambda_2)$ be a connected partition of $G$ where all parts are more than one. We will do some calculations for $[e_{\lambda}]X_G$ considering $X_G$ written in the power-sum basis. Because the power-sum and elementary bases are multiplicative and because of equation~\eqref{eq:powersum_to_elem} we can see that $e_{(n)}$ will only appear in $p_{\nu}$ when $\nu=(n)$. This means that 
\begin{equation}
[e_{(n)}]X_G=[p_{(n)}]X_G\cdot [e_{(n)}]p_{(n)}=(-1)^{n-1}n\cdot [p_{(n)}]X_G.
\label{eq:elem_conversion1}
\end{equation}
By equation~\eqref{eq:powersum_to_elem}
we can also see that that $e_{\lambda}$ will only appear in $p_{\nu}$ when $\nu=(n)$ or $\nu=\lambda$ so
\begin{equation}
[e_{\lambda}]X_G=[p_{(n)}]X_G\cdot [e_{\lambda}]p_{(n)}+[p_{\lambda}]X_G\cdot [e_{\lambda}]p_{\lambda}=[p_{(n)}]X_G\cdot [e_{\lambda}]p_{(n)}+[p_{\lambda}]X_G\cdot (-1)^{n-2}\lambda_1\lambda_2.
\label{eq:elem_conversion2}
\end{equation}

By the signs of the terms in equation~\eqref{eq:X_Gpowersum_poset} we know that $ (-1)^{n-2}[p_{\lambda}]X_G>0$ so 
$$[e_{\lambda}]X_G>[p_{(n)}]X_G\cdot [e_{\lambda}]p_{(n)}.$$
Using equation~\eqref{eq:powersum_to_elem} and equation~\eqref{eq:elem_conversion1} this gives 
$$[e_{\lambda}]X_G>-[e_{(n)}]X_G$$
when $\lambda_1\neq \lambda_2$ and 
$$[e_{\lambda}]X_G>-\frac{[e_{(n)}]X_G}{2}$$
when $\lambda_1=\lambda_2$, which completes the proof. 
\end{proof}

Now we just need some coefficients for the chromatic symmetric functions for stars $S_n$. 

\begin{lemma}
Let $\lambda=(\lambda_1,\lambda_2)$ be an integer partition with parts greater than one. Then, 
\begin{enumerate}[(i)]
\item $[e_{\lambda}]X_{S_n}=-n$ if $\lambda_1\neq \lambda_2$ or
\item $[e_{\lambda}]X_{S_n}=-\frac{n}{2}$ if $\lambda_1= \lambda_2$.
\end{enumerate}
\label{lem:star_coeff}
\end{lemma}
\begin{proof}
Note that there does not exist a connected partition for $S_n$ of type $\lambda=(\lambda_1,\lambda_2)$ as described, so by equation~\eqref{eq:X_Gpowersum_poset} we can see that $[p_{\lambda}]X_{S_n}=0$. Because equation~\eqref{eq:elem_conversion2} is always true,  using this equation and equations~\eqref{eq:powersum_to_elem} and~\eqref{eq:elem_conversion1} we can calculate 
$$[e_{\lambda}]X_{S_n}=-n$$ when $\lambda_1\neq \lambda_2$ and
\item $$[e_{\lambda}]X_{S_n}=-\frac{n}{2}$$ 
when $\lambda_1= \lambda_2$.
\end{proof}

Now we have all the tools we need to show that stars are independent in ${\mathcal E}_n$. 

\begin{proposition}
The star $S_n$ is an independent element in  ${\mathcal E}_n$ for $n\geq 4$. 
\label{prop:stars_independent}
\end{proposition}
\begin{proof}
Since all trees are maximal elements by Proposition~\ref{prop:trees_maximal_elem} it suffices to show that stars are minimal elements. Assume that $G$ is a connected graph on $n$ vertices with $S_n>_e G$. We will show $S_n>_e G$ is impossible by showing $X_e(S_n,G)$ is not $e$-positive. The graph $G$ can not be a tree because all trees are maximal, so by Lemma~\ref{lem:nontrees} we know that $G$ has a connected partition of type $\lambda=(\lambda_1,\lambda_2)$ with all parts greater than one. Let us first consider the case when $\lambda_1\neq \lambda_2$. We have in this case using the coefficients calculated in Lemma~\ref{lem:star_coeff} and Corollary~\ref{cor:X_T_coeff} that
$$[e_{\lambda}]X_e(S_n,G)=-n-\frac{n}{[e_{(n)}]X_G}\cdot [e_{\lambda}]X_G.$$
Lemma~\ref{lem:nontrees} implies that this coefficient is negative so $X_e(S_n,G)$ is not $e$-positive in this case and we have a contradiction. Next let us consider the case when $\lambda_1= \lambda_2$. Using the coefficients calculated in Lemma~\ref{lem:star_coeff} and Corollary~\ref{cor:X_T_coeff} we have that
$$[e_{\lambda}]X_e(S_n,G)=-\frac{n}{2}-\frac{n}{[e_{(n)}]X_G}\cdot [e_{\lambda}]X_G.$$
Lemma~\ref{lem:nontrees} implies that this coefficient is negative so $X_e(S_n,G)$ is not $e$-positive in all cases. Thus,  stars are minimal and further are independent elements in $\cE_n$. 
\end{proof}

\begin{corollary}
The poset  ${\mathcal E}_n$ is not a lattice for $n\geq 4$. 
\end{corollary}
\begin{proof}
Lattices do not have independent elements, and we have shown by  Proposition~\ref{prop:stars_independent} that the poset ${\mathcal E}_n$ for $n\geq 4$ has an independent element. 
\end{proof}

Our last result about $\cE_n$ we will prove in this section is that the family of lollipop graphs form a chain with the complete graph as the minimal element and the path as the maximal element. To prove this we will use the following formulas for the chromatic symmetric functions and polynomials of lollipop graphs. 

\begin{theorem}[\cite{DvW} Theorem 7 and Lemma 11]
The chromatic polynomial of a lollipop $L_{m,n}$  is
$$\chi_{L_{m,n}}(k)=k(k-1)^{n+1}(k-2)(k-3)\cdots (k-(m-1))$$
and the chromatic symmetric function satisfies
$$X_{L_{m,n}}=(m-1)X_{L_{m-1,n+1}}-(m-2)X_{K_{m-1}}X_{P_{n+1}}.$$
\label{thm:lollipop}
\end{theorem}

The above theorem can be applied to give us the coefficient of $e_{(m+n)}$ for a lollipop graph. 

\begin{proposition}The coefficient of $e_{(m+n)}$ for the lollipop graph is
$$[e_{{(m+n)}}]X_{L_{m,n}}=(m+n) (m-1)!.$$
\label{prop:lollipop_coeff}
\end{proposition}
\begin{proof}
By equation~\ref{eq:e_n_coeff} we can use the chromatic polynomial of $L_{m,n}$ given in Theorem~\ref{thm:lollipop} to calculate the coefficient as
\begin{align*}
[e_{{(m+n)}}]X_{L_{m,n}}&=(-1)^{m+n-1}(m+n) \cdot [k]\left(k(k-1)^{n+1}(k-2)(k-3)\cdots (k-(m-1))\right)\\
&=(m+n) (m-1)!
\end{align*}
so we are done.
\end{proof}

\begin{theorem}
The family of lollipop graphs $\{L_{m,n}\}$ on $m+n=N$ vertices forms a chain in  ${\mathcal E}_N$. In particular, for $m\geq 3$ and $n\geq 0$,
$$L_{m-1,n+1}\geq_e L_{m,n}.$$
Hence, the path $P_N$ is the maximal element and the complete graph $K_N$ is the minimal element of the chain. 
\label{thm:lollipop_e-chain}
\end{theorem}
\begin{proof}
It suffices to show $L_{m-1,n+1}\geq_e L_{m,n}$ by transitivity for $m\geq 3$ and $n\geq 0$. Using the coefficients calculated in Proposition~\ref{prop:lollipop_coeff} and the formula in Theorem~\ref{thm:lollipop} we have
\begin{align*}
X_e(L_{m-1,n+1}, L_{m,n})&=X_{L_{m-1,n+1}}-\frac{1}{m-1}X_{L_{m,n}}\\
&=X_{L_{m-1,n+1}}-\frac{1}{m-1}\left((m-1)X_{L_{m-1,n+1}}-(m-2)X_{K_{m-1}}X_{P_{n+1}}\right)\\
&=\frac{m-2}{m-1}X_{K_{m-1}}X_{P_{n+1}},
\end{align*}
which is known to be $e$-positive by Example~\ref{ex:complete_graph},  Remark~\ref{re:e-pos_graphs} and that the multiplication of two $e$-positive functions is $e$-positive.  \end{proof}

\section{Properties of the chromatic Schur-positivity poset}
\label{sec:schur}

In this section we determine properties of $\cS_n$.  These properties parallel the properties we have already proven for $\cE_n$. In Theorem~\ref{thm:G>K_n_schur} we proved that $G\geq_s K_n$ if and only if $G$ is Schur-positive, and we will further prove that $K_n$ is  a minimal element. We will find that as we increase in $\cS_n$ that the   independence number, $\alpha(G)$, increases, the number of acyclic orientations decreases, and the chromatic number decreases. Also, we will find that the trees with distinct chromatic symmetric functions form an anti-chain and are maximal elements. In particular, the stars $S_n$ are independent elements, so $\cS_n$ also cannot be a lattice. Lastly, we prove that the family of  lollipop graphs is a chain with the complete graph as the minimal element and the path as the maximal element, however, the proof is more complex than in the case of $\cE_n$. We present the theorem and proof in Section~\ref{sec:schur_lollipop}. 

First we consider several common statistics on graphs and discover if there is a consistent relationship to the relations in $\cS_n$. 

\begin{proposition}
If $G\geq_s H$ then $\alpha(G)\geq \alpha(H)$.
\label{prop:alpha_schur}
\end{proposition}
\begin{proof}
Suppose $\alpha(G) < \alpha(H)$ and let $\lambda$ be the partition of $n$ given by $(\alpha(H), 1^{n-\alpha(H)}).$ By definition, $H$ has a  stable partition of type $\lambda$ but $G$ does not. Hence, by the expansion of $X_G$ and $X_H$ into the monomial basis, Theorem~\ref{thm:X_Gmonomial}, we have $[m_\lambda] X_G = 0$ but $[m_\lambda] X_H \geq 1$ since the coefficient of $m_\lambda$ in $X_G$  is a multiple of the number of stable partitions of type $\lambda$ in $G$. Therefore, $[m_\lambda] X_s(G,H) < 0$ since Theorem~\ref{thm:s_hook_coeff} implies that the scaling factor is always positive. Hence by Remark~\ref{remark:mpos}, $X_s(G,H)$ cannot be Schur-positive as $X_s(G,H)$ is not $m$-positive. 
\end{proof}

\begin{proposition}
If $G>_s H$ then the number of acyclic orientations for $G$ is less than the number of acyclic orientations for $H$.
\end{proposition}
\begin{proof}
Let $G>_s H$ so $X_s(G,H)\neq 0$ is Schur-positive. By Theorem~\ref{thm:X_Gpowersum} we can see that $[p_{(1^n)}]X_G=[p_{(1^n)}]X_H=1$ so we have that
$$[p_{(1^n)}]X_s(G,H)=1-\frac{[s_{(1^n)}]X_G}{[s_{(1^n)}]X_H}.$$
By Lemma~\ref{lem:conversion_note} this coefficient is positive so $[s_{(1^n)}]X_G<[s_{(1^n)}]X_H$. By Theorem~\ref{thm:s_hook_coeff} we have our result. 
\end{proof}

\begin{proposition}
If $G \geq_s H$, then $\chi(G) \leq \chi(H)$.
\end{proposition}

\begin{proof}
Suppose $\chi(G) > \chi(H)$. Since there is a coloring of $H$ with $\chi(H)$ colors, then $H$ has a stable partition of some type $\lambda \vdash n$ with length $\chi(H)$ 
and hence, we have $[m_\lambda] X_H \geq 1$ by Theorem~\ref{thm:X_Gmonomial}. Furthermore, $[m_\lambda] X_G = 0$ since by definition, $G$ cannot be colored with fewer than $\chi(G)$ colors and by assumption $\chi(G) > \chi(H)$. Therefore, $[m_\lambda] X_s(G,H) < 0$ since Theorem~\ref{thm:s_hook_coeff} implies that the scaling factor is always positive. Hence by Remark~\ref{remark:mpos}, $X_s(G,H)$ cannot be Schur-positive  as  $X_{{s}}(G,H)$ is  not $m$-positive. 
\end{proof}

\begin{remark}
Note that there is no consistent relationship between  clique numbers $\omega(G)$ and relations in the poset $\cS_n$. We find that $P_5>_s L_{3,2}>_s C_5$ but $\omega(P_5)<\omega(L_{3,2})>\omega(C_5)$. 
\end{remark}

We have proven in Theorem~\ref{thm:G>K_n_schur} that $G$ is Schur-positive if and only of $G\geq_s K_n$.  Now we will prove that the complete graph is a minimal element. 

\begin{corollary}
The complete graph $K_n$ is a minimal element in ${\mathcal S}_n$. 
\end{corollary}
\begin{proof}
If $G\neq K_n$ is a connected graph on $n$ vertices then there are two vertices $u$ and $v$ without an edge between them. The set $\{u,v\}$ is an independent set and  $\alpha(G)\geq 2$. By Proposition~\ref{prop:alpha_schur} we know $K_n\not\geq_s G$ because $1=\alpha(K_n)<\alpha(G)$. Hence $K_n$ is a minimal element. 
\end{proof}

Similar to Theorem~\ref{thm:elem_antichain} there is an elegant condition that can generate many  anti-chains in $\cS_n$. 

\begin{theorem}
Let $\{G_1, G_2, \ldots, G_k\}$ be some set of connected graphs on $n$ vertices with equal $s_{(1^n)}$ coefficients, $[s_{(1^n)}]X_{G_i}=[s_{(1^n)}]X_{G_j}$, and distinct chromatic symmetric functions. Then 
$\{G_1, G_2, \ldots, G_k\}$ is an anti-chain  in  ${\mathcal S}_n$. 
\label{thm:schur_antichain}
\end{theorem}

\begin{proof}
Let $\{G_1, G_2, \ldots, G_k\}$ satisfy all the assumptions stated in the proposition. This means $X_s(G_i,G_j)=X_{G_i}-X_{G_j}$. By Theorem~\ref{thm:Schur_X_G-X_H} we know that either that $X_{G_i}-X_{G_j}=0$ so $i=j$ or that $X_{G_i}-X_{G_j}$ is not Schur-positive. This implies that $G_i$ is not related to $G_j$ for all $i\neq j$, which means that $\{G_1, G_2, \ldots, G_k\}$ is an anti-chain  in  ${\mathcal S}_n$. 
\end{proof}

 Using this proposition we can prove that trees form an anti-chain. 
 
 \begin{corollary}
Trees on $n$ vertices with distinct chromatic symmetric functions form an anti-chain  in  ${\mathcal S}_n$.
\end{corollary}
\begin{proof}
By Corollary~\ref{cor:X_T_coeff} all trees $T$ on $n$ vertices have $[s_{(1^n)}]X_T=2^{n-1}$, so by Theorem~\ref{thm:schur_antichain} we have our result. 
\end{proof}

\begin{remark}
By Theorem~\ref{thm:schur_antichain} given any integer $z\in\Z$ the collection of graphs on $n$ vertices $\{G:[s_{(1^n)}]X_G=z\}$ is an anti-chain in $\cS_n$ under the assumption that we are grouping graphs together in $\cS_n$ if they have equal chromatic symmetric function. 
\end{remark}

Trees are maximal elements in $\cS_n$ just like they are in $\cE_n$. Before we prove this we will prove a lemma on the number of acyclic orientations of non-tree graphs compared with trees. 

\begin{lemma}
Trees $T$ on $n$ vertices have $2^{n-1}$ acyclic orientations. Graphs $G$ on $n$ vertices that are not trees have more than $2^{n-1}$ acyclic orientations. 
\label{lem:acyclic_bound}
\end{lemma}
\begin{proof}
This can be proved using deletion-contraction and induction on the number of vertices and edges in graphs. Our base case is any tree on $n\geq 1$ vertex, which will have $n-1$ edges. Since there are no cycles there are $2^{n-1}$ acyclic orientations. Now let $G$ be a non-tree connected graph on $n\geq 2$ vertices, so $G$ has at least $n$ edges. There exists an edge $\epsilon\in E(G)$ where $G-\epsilon$ is still connected. By deletion-contraction  we have that 
$\chi_G(k)=\chi_{G-\epsilon}(k)-\chi_{G/\epsilon}(k).$
Because equation~\eqref{eq:total_acyclic} tells us that $(-1)^{n}\chi_G(-1)$ is the total number of acyclic orientations for $G$. Using induction we get that the total number of acyclic orientations for $G$ is 
$$
(-1)^n\chi_G(-1)=(-1)^n\chi_{G-\epsilon}(-1)+(-1)^{n-1}\chi_{G/\epsilon}(-1)\geq 2^{n-1}+2^{n-2},
$$
which is certainly greater than $2^{n-1}$. 
\end{proof}

\begin{proposition}
 All trees on $n$ vertices are maximal elements  in  ${\mathcal S}_n$.
 \label{prop:trees_Schur_max}
\end{proposition}
\begin{proof}
Suppose that $G>_s T$ for some tree $T$ and connected graph $G$ on $n$ vertices. We will show that this leads to a contradiction. By Lemma~\ref{lem:conversion_note} if $F\in\Lambda^n$ is a non-zero Schur-positive function then $[p_{(1^n)}]F > 0$. Since $G>_s T$ we know that $X_s(G,T)$ is a non-zero Schur-positive function so $[p_{(1^n)}]X_s(G,T)> 0$. By Theorem~\ref{thm:X_Gpowersum} we can see that $[p_{(1^n)}]X_G=[p_{(1^n)}]X_T=1$ so using  Corollary~\ref{cor:X_T_coeff} we have that
$$[p_{(1^n)}]X_s(G,T)=1-\frac{[s_{(1^n)}]X_G}{2^{n-1}}>0.$$
This means that $2^{n-1}>[s_{(1^n)}]X_G$ where $[s_{(1^n)}]X_G$ is the total number of acyclic orientations of $G$ by Theorem~\ref{thm:s_hook_coeff}. By Lemma~\ref{lem:acyclic_bound} this is a  contradiction. 
\end{proof}

Just like in $\cE_n$, while all trees are maximal in $\cS_n$ some are actually independent elements. The stars $S_n$ are a family of independent elements. In order to prove this we need to study some specific coefficients of $X_{S_n}$ and $X_G$ for a general graph $G$.  We will need the conversion formula from monomial symmetric functions to Schur symmetric functions. 
 In Macdonald's book~\cite[page 105]{M79} the transition formula  is $$m_{\lambda} =\sum_{\nu\vdash n}K_{\lambda,\nu}^{-1}s_{\nu}.$$
The coefficients  $K_{\lambda,\nu}^{-1}$ are the {\it inverse Kostka numbers} define by 
\begin{equation}
K_{\lambda,\nu}^{-1}=\sum_T (-1)^{\text{ht}(T)},
\label{eq:inverse_kotska}
\end{equation}
which is summed over special rim hooks $T$ with  underlying Young diagram  $\nu$.  See~\cite[page 107]{M79} for full details. Particularly, Macdonald's book notes that $K_{\lambda,\mu}^{-1}=0$ unless $\lambda \succeq\nu$ in dominance order. 

\begin{lemma} Let $n\geq 4$. Then
\begin{enumerate}[(i)]
\item $[s_{(n-2,2)}]X_G=[m_{(n-2,2)}]X_G-[m_{(n-1,1)}]X_G$ and
\item $[s_{(n-2,2)}]X_{S_n}=-1$. 
\item If $\alpha(G)\leq n-2$ then $[s_{(n-2,2)}]X_G\geq 0$. 
\item If $\alpha(G)=n-1$ then $G=S_n$. 
\end{enumerate}
\label{lem:Schur_stars}
\end{lemma}
\begin{proof}
From the definition of inverse Kostka numbers given below equation~\eqref{eq:inverse_kotska} we can see that $K_{(n-2,2),(n-2,2)}^{-1}=1$, $K_{(n-1,1),(n-2,2)}^{-1}=-1$ and $K_{(n),(n-2,2)}^{-1}=0$. 
Since $K_{\lambda,\mu}^{-1}=0$ unless $\lambda \succeq\mu$ we know that $s_{(n-2,2)}$ only appears in the expansion of $m_{\lambda}$ in the Schur basis when $\lambda\succeq (n-2,2)$. Such $\lambda$ are $(n)$, $(n-1,1)$ and $(n-2,2)$. So using the expansion of $X_G$ in the monomial basis we can calculate the coefficient of $s_{(n-2,2)}$, which is
\begin{align*}
[s_{(n-2,2)}]X_G&=K_{(n-2,2),(n-2,2)}^{-1}\cdot [m_{(n-2,2)}]X_G+K_{(n-1,1),(n-2,2)}^{-1}\cdot [m_{(n-1,1)}]X_G+K_{(n),(n-2,2)}^{-1}\cdot [m_{(n)}]X_G\\
&= [m_{(n-2,2)}]X_G-[m_{(n-1,1)}]X_G.
\end{align*}

Note that the star $S_n$ has exactly one stable partition of type $(n-1,1)$ and does not have any stable partitions of type $(n-2,2)$. By Theorem~\ref{thm:X_Gmonomial} we have 
$[m_{(n-2,2)}]X_{S_n}=0$ and $[m_{(n-1,1)}]X_{S_n}=1$ so $[s_{(n-2,2)}]X_{S_n}=-1$. 
Further note that if $\alpha(G)\leq n-2$ there is no stable partition of $G$ of type $(n-1,1)$, so $[s_{(n-2,2)}]X_G=[m_{(n-2,2)}]X_G\geq 0$. Lastly note that in the case where $G$ is a connected graph with $\alpha(G)= n-1$ there are $n-1$ vertices with no edges between them and at most $n-1$ edges from these $n-1$ vertices to the last $n$th vertex. For the graph to be connected we need all $n-1$ edges, which describes the star $S_n$. 
\end{proof}

We now have everything we need to prove stars are independent elements. 

\begin{proposition}
The star $S_n$ is an independent element in  ${\mathcal S}_n$ for $n\geq 4$. 
\label{prop:stars_independent_Schur}
\end{proposition}
\begin{proof}
Let $G$ be a connected graph that is not a star. By Proposition~\ref{prop:trees_Schur_max} we know that $S_n$ is a maximal element in ${\mathcal S}_n$ so we only have to show that {$S_n\not \geq_s G$} for any $G\neq S_n$.
Because connected graphs can have at most $n-1$ independent vertices, by Lemma~\ref{lem:Schur_stars} we know that $\alpha(G)\leq n-2$ and so $[s_{(n-2,2)}]X_G\geq 0$. Recall that by Theorem~\ref{thm:s_hook_coeff} that $[s_{(1^n)}]X_G> 0$ because $[s_{(1^n)}]X_G$ counts the number of acyclic orientations of $G$. Using Corollary~\ref{cor:X_T_coeff} and Proposition~\ref{prop:trees_Schur_max} we have  that 
$$[s_{(n-2,2)}]X_s(S_n,G)=-1-2^{n-1}\frac{[s_{(n-2,2)}]X_G}{[s_{(1^n)}]X_G}<0,$$
which shows that  $S_n\not \geq_s G$. \end{proof}

\begin{corollary}
The poset ${\mathcal S}_n$ is not a lattice for $n\geq 4$. 
\end{corollary}
\begin{proof}
Lattices do not have independent elements, which by  Proposition~\ref{prop:stars_independent_Schur} the poset ${\mathcal S}_n$ has. 
\end{proof}

In $\cS_n$, similar to $\cE_n$, we have that the family of lollipop graphs on $n$ vertices forms a chain with the complete graph as the minimal element and the path as the maximum element. The proof is more complex than in the case of $\cE_n$, so we present the proof in its own section, concluding with the theorem in Theorem~\ref{the:lollipop}. 

\section{Lollipops in the chromatic Schur-positivity poset}
\label{sec:schur_lollipop}

In this section we prove that the set of lollipops $\{L_{m,n}:m+n=N\}$ forms a chain in the poset $\cS_N$. However, the proof will not be as  straightforward as it was in the case of $\cE_N$. We will use Gasharov's~\cite{Gasharov}  interpretation for the coefficients of $X_G$  in the Schur basis. This interpretation is in  terms of $P$-tableau  in the case when $G$ is an incomparability graph of $(3+1)$-free poset. All lollipop graphs are examples of incomparability graphs of $(3+1)$-free posets. Before we present the needed background on $P$-tableaux we will set up our proof. First we will need the coefficient $[s_{(1^{m+n})}]X_{L_{m,n}}$. 

\begin{proposition} We have
$[s_{(1^{m+n})}]X_{L_{m,n}}=2^nm!$.
\label{prop:lollipop_coeff_schur}
\end{proposition}
\begin{proof}
Using Theorem~\ref{thm:lollipop} and equation~\eqref{eq:s_1^n_coeff} we have that 
$$[s_{(1^{m+n})}]X_{L_{m,n}}=(-1)^{m+n}\chi_{L_{m,n}}(-1)=2^nm!,$$
which completes the proof.
\end{proof}

\begin{figure}
\begin{tikzpicture}
\coordinate (A) at (0,0);
\coordinate (B) at (1,-1);
\coordinate (C) at (1,1);
\coordinate (D) at (2,0);
\coordinate (E) at (3,0);
\coordinate (F) at (4,0);
\coordinate (G) at (5,0);
\draw[black] (A)--(B)--(C)--(D)--(A)--(C);
\draw[black] (B)--(D)--(E)--(F)--(G);
\filldraw[black] (A) circle [radius=2pt] node[below] {$1$};
\filldraw[black] (B) circle [radius=2pt] node[right] {$2$};
\filldraw[black] (C) circle [radius=2pt] node[right] {$3$};
\filldraw[black] (D) circle [radius=2pt] node[below] {$4$};
\filldraw[black] (E) circle [radius=2pt] node[below] {$5$};
\filldraw[black] (F) circle [radius=2pt] node[below] {$6$};
\filldraw[black] (G) circle [radius=2pt] node[below] {$7$};

\begin{scope}[shift={(7,-1)}]
\coordinate (A) at (0,0);
\coordinate (B) at (1,0);
\coordinate (C) at (2,0);
\coordinate (D) at (3,0);
\coordinate (E) at (0,1);
\coordinate (F) at (1,1);
\coordinate (G) at (1,2);
\draw[black] (A)--(E)--(G)--(D)--(F)--(C)--(E)--(B);
\draw[black] (A)--(F)--(B);
\filldraw[black] (A) circle [radius=2pt] node[below] {$1$};
\filldraw[black] (B) circle [radius=2pt] node[below] {$2$};
\filldraw[black] (C) circle [radius=2pt] node[below] {$3$};
\filldraw[black] (D) circle [radius=2pt] node[below] {$4$};
\filldraw[black] (E) circle [radius=2pt] node[left] {$5$};
\filldraw[black] (F) circle [radius=2pt] node[left] {$6$};
\filldraw[black] (G) circle [radius=2pt] node[left] {$7$};
\end{scope}
\end{tikzpicture}
\caption{On the left we have $L_{4,3}$ and on the right we  have the Hasse  diagram of its associated poset  $\cP_{4,3}$.}
\label{fig:poset}
\end{figure}
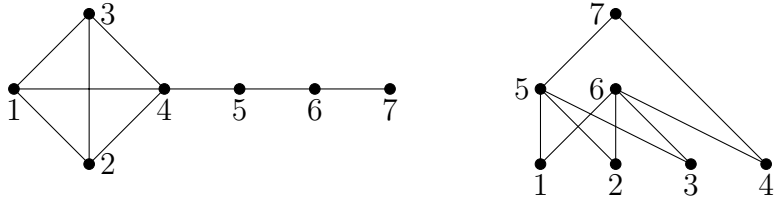

To prove that the lollipops form a chain it suffices to show that $L_{m,n}\geq_s L_{m+1,n-1}$ for any $m\geq 2$ and $n\geq 1$. Using Proposition~\ref{prop:lollipop_coeff_schur} this is equivalent to showing that 
$$X_s(L_{m,n}, L_{m+1,n-1})=X_{L_{m,n}}-\frac{2}{m+1}X_{L_{m+1,n-1}}$$
is Schur-positive. 
It suffices to show that 
$$2\cdot [s_{\lambda}]X_{L_{m+1,n-1}}\leq (m+1)\cdot [s_{\lambda}]X_{L_{m,n}}.$$

Now we will  introduce $P$-tableaux in the case of lollipops. See~\cite{Gasharov} for more details.  
Consider the poset $\cP_{m,n}$ on $[m+n]$ where $i\leq_{\cP_{m,n}} j$ if and only if $i\leq j$ and ${ij}\notin E(L_{m,n})$. The lollipop $L_{m,n}$ is the incomparability graph of $\cP_{m,n}$. See Figure~\ref{fig:poset}. 
A {\it P-tableau} of shape $\lambda\vdash m+n$ for  $L_{m,n}$ is a filling of the Young diagram of $\lambda$. We will use the convention of drawing our Young diagrams of shape $\lambda=(\lambda_1,\lambda_2,\ldots, \lambda_{{l(\lambda)}})$ so that row 1 is at the top with $\lambda_1$ boxes, row ${l(\lambda)}$ is at the bottom with  $ \lambda_{{l(\lambda)}}$ boxes and all  rows are left-justified. We will number  columns from left to right.  
 We fill the Young diagram with $[m+n]$ so that:
\begin{enumerate}
\item The rows are  increasing with respect to the poset $\cP_{m,n}$. 
\item There are no adjacent decreases along the columns with respect to the poset $\cP_{m,n}$. 
\end{enumerate}
This means that if $i$ appears to the left of $j$ in the same row then $i<_{\cP_{m,n}} j$ and if $i$ appears immediately above $j$ in the same column then $i\not>_{\cP_{m,n}} j$. Let $\cT_{\lambda,m,n}$ be the collection of all $P$-tableau for  $L_{m,n}$ of shape $\lambda$.  Gasharov's result states that 
$$[s_{\lambda}]X_{L_{m,n}}=\# \cT_{\lambda,m,n}$$
so we are done if we can show that $$2\cdot \#\cT_{\lambda,m+1,n-1}\leq (m+1)\cdot \#\cT_{\lambda,m,n}.$$We will show the inequality above by defining an injection
$$f_{\lambda}:[2]\times \cT_{\lambda,m+1,n-1}\rightarrow [m+1]\times \cT_{\lambda,m,n}.$$
To more easily compare $P$-tableaux for $L_{m,n}$ and $L_{m+1,n-1}$ we will relabel the vertices  $[m+n]$ with  $\cA=\{A_1,A_2,\ldots ,A_{m}\}$, $C$ and $\cB=\{B_1,B_2,\ldots, B_{n-1}\}$. In $L_{m+1,n-1}$ the vertices in $\cA$ are those in $K_{m+1}$ not adjacent to the path, $C$ is the vertex in $K_{m+1}$ adjacent to the path and the path is labeled with $\cB$ so that the smaller subscripts are closer to $C$.  In $L_{m,n}$ the vertices in $\cA-\{A_m\}$ are those in $K_{m}$ not adjacent to the path, $A_m$ is the vertex in $K_{m}$ adjacent to the path and the path is labeled with $\cB\cup\{C\}$ so that $C$ is adjacent to $A_m$ and the smaller subscripts in $\cB$ are closer to $C$.
See Figure~\ref{fig:lollipop_labeling}. 

\begin{figure}
\begin{tikzpicture}
\coordinate (A) at (0,0);
\coordinate (B) at (1,-1);
\coordinate (C) at (1,1);
\coordinate (D) at (2,0);
\coordinate (E) at (3,0);
\coordinate (F) at (4,0);
\coordinate (G) at (5,0);
\draw[black] (A)--(B)--(C)--(D)--(A)--(C);
\draw[black] (B)--(D)--(E)--(F)--(G);
\filldraw[black] (A) circle [radius=2pt] node[below] {$A_1$};
\filldraw[black] (B) circle [radius=2pt] node[right] {$A_2$};
\filldraw[black] (C) circle [radius=2pt] node[right] {$A_3$};
\filldraw[black] (D) circle [radius=2pt] node[below] {$C$};
\filldraw[black] (E) circle [radius=2pt] node[below] {$B_1$};
\filldraw[black] (F) circle [radius=2pt] node[below] {$B_2$};
\filldraw[black] (G) circle [radius=2pt] node[below] {$B_3$};

\begin{scope}[shift={(7,0)}]
\coordinate (A) at (0,0);
\coordinate (B) at (1,-1);
\coordinate (C) at (1,1);
\coordinate (D) at (2,0);
\coordinate (E) at (3,0);
\coordinate (F) at (4,0);
\coordinate (G) at (5,0);
\draw[black] (A)--(B)--(C)--(A);
\draw[black] (C)--(D)--(E)--(F)--(G);
\filldraw[black] (A) circle [radius=2pt] node[below] {$A_1$};
\filldraw[black] (B) circle [radius=2pt] node[right] {$A_2$};
\filldraw[black] (C) circle [radius=2pt] node[right] {$A_3$};
\filldraw[black] (D) circle [radius=2pt] node[below] {$C$};
\filldraw[black] (E) circle [radius=2pt] node[below] {$B_1$};
\filldraw[black] (F) circle [radius=2pt] node[below] {$B_2$};
\filldraw[black] (G) circle [radius=2pt] node[below] {$B_3$};
\end{scope}
\end{tikzpicture}
\caption{For $m=3$ and $n=4$ on the left we have $L_{m+1,n-1}$ and on the right we have  $L_{m,n}$ using the vertex labels $\cA$, $\cB$ and $C$.}
\label{fig:lollipop_labeling}
\end{figure}
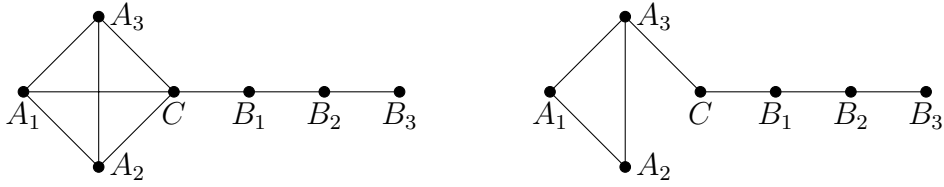

With these labels we can now more specifically describe the $P$-tableaux, whose proof follows by definition. 
\begin{lemma}The following are the rules for adjacent cells in the $P$-tableaux of $L_{m+1,n-1}$. 
\begin{enumerate}
\item Let $X$ be immediately left of $Y$. 
\begin{enumerate}[(i)]
\item If $X\in\cA$ then $Y\in \cB$. 
\item If $X=C$ then $Y\in \cB-\{B_1\}$. 
\item If $X=B_i$ then $Y=B_j$ with $j\geq i+2$. 
\end{enumerate}
\item Let $X$ be immediately above of $Y$. 
\begin{enumerate}[(i)]
\item If $X\in\cA$ then $Y$ can have any label. 
\item If $X=C$ then $Y$ can have any label. 
\item If $X=B_i$, $i\geq 1$, then $Y=B_j$ with $j\geq i-1$ considering $C=B_0$. 
\end{enumerate}
\end{enumerate}
\label{lem:Ptableaux_rules}
\end{lemma}

We now have all the background we need to define our injection $f_{\lambda}$, but in order to make our argument smoother we will first establish some facts about $P$-tableaux in the following structure lemma.   See Figure~\ref{fig:def_ashift} for a concrete illustration.

\begin{lemma} We have the following for $T\in \cT_{\lambda,m+1,n-1}$. 
\begin{enumerate}[(a)]
\item All of the fillings from $\cA\cup\{C\}$ appear in the first column. 
\item The tableau $T\notin \cT_{\lambda,m,n}$ if and only if $C$ appears directly above some $A_j$ with $j\in[m-1]$.
\item  Column 1 of $T$ is composed of, reading from top to bottom, $\cA'$, $\cB'$, $C$, $\cA''$ and $\cB''$ where 
$\cA'$ and $\cA''$ are  contiguous blocks of cells with fillings from $\cA$ and $\cB'$ and $\cB''$ are contiguous blocks of cells  containing fillings from $\cB$, any of which could possibly be empty. Particularly, if $\cB'$ is non-empty then $\cB'$ contains $\{B_1,B_2,\ldots, B_s\}$ with subscripts increasing as you go up in $T$. 
\item If $B_s$ appears in column $c$ of $T$ and $B_{s+1}$ appears in column $c+1$ below $B_s$ then above and including $B_{s+1}$ appear all of $B_i$, $i\in[s+1,s+L]$ for some $L>1$, in a contiguous block of cells with subscripts increasing as we go up in $T$.  To the immediate right of $B_s$ we have one of the  $B_i$, $i\in[s+1,s+L]$. Additionally, any $B_i$ above $B_{s+L}$  in column $c+1$ has $i<s$ and any $B_i$ below $B_{s+1}$ in column $c+1$  has $i>s+L$. 
\end{enumerate}
\label{lem:Ptableaux_structure}
\end{lemma}
\begin{proof}
Part (a)  follows from the fact that $\cA\cup\{C\}$ are minimal elements in $\cP_{m+1,n-1}$. 

Part (b) comes from the fact that $L_{m,n}$ and $L_{m+1,n-1}$ have the exact same edges except $L_{m,n}$ is missing the edges between the  vertices in $\cA-\{A_m\}$ and $C$. This means that if $T\in \cT_{\lambda,m+1,n-1}$ then the only relations that could disrupt $T\in \cT_{\lambda,m,n}$ is $C>_{\cP_{m,n}}A_i$ for all $i\in[m-1]$. Meaning that $T\notin \cT_{\lambda,m,n}$ if and only if $C$ is above $A_j$ for some $j\in [m-1]$. 

Part  (c) follows from Lemma~\ref{lem:Ptableaux_rules}. The contiguous block breaks down particularly follows from the fact that directly below a filling from $\cB$ we can only have another filling from $\cB$ or the filling $C$. After this filling $C$ we could have some fillings from $\cA$, but as soon as there is one more filling from $\cB$ then we only have fillings from $\cB$. Now we will use the labels for the contiguous blocks of cells described in part (c) of this lemma. From the information we have we can conclude  exactly what the fillings of $\cB'$ are if $\cB'$  is non-empty. Because $\cB'$ is above the filling $C$ by Lemma~\ref{lem:Ptableaux_rules} the only possibility is that it contains $\{B_1,B_2,\ldots, B_s\}$ with subscripts increasing as you go up in $T$. 

For part (d) assume that $B_s$ appears in column $c$ of $T$ and $B_{s+1}$ appears in column $c+1$ below $B_s$ in a lower row. We know for sure by Lemma~\ref{lem:Ptableaux_rules} that to the right of $B_s$ we have $B_t$ where $t\geq s+2$. The only way to have $B_t$ in column $c+1$ somewhere above $B_{s+1}$ is to have $t>s+1$ and for there to be all of the $B_i$, $i\in [s+1,t]$, between $B_{s+1}$ and $B_t$ with subscripts increasing as we go up in $T$. This pattern of  fillings from $\cB$ with increasing subscripts may continue beyond $B_t$ in consecutive cells until some highest $B_{s+L}$. Because we know the placement for all $B_i$ for $i\in[s,s+L]$ by Lemma~\ref{lem:Ptableaux_rules} we can conclude that all $B_i$ above $B_{s+L}$ in column $c+1$ have $i<s$ and any $B_i$ below $B_{s+1}$ in column $c+1$  have $i>s+L$.
\end{proof}

Now we will define our injection $f_{\lambda}:[2]\times \cT_{\lambda,m+1,n-1}\rightarrow [m+1]\times \cT_{\lambda,m,n}$. Let $T\in \cT_{\lambda, m+1,n-1}$ and $k\in [2]$. 

{\bf Case 1:} Say $T\in \cT_{\lambda, m,n}$ and $k\in[2]$. We map
$$f_{\lambda}(k,T)= (m-1+k,T).$$
See Figure~\ref{fig:cases1and2} for an example. 
This case is clearly well defined and injective. Additionally, note that  the first coordinate of the output is either $m$ or $m+1$. We will see in Case 2 and Case 3 of our map that the first coordinate of the output will be at most  $m-1$, so will not intersect  Case 1. 

{\bf Case 2:} Say that instead $T\notin \cT_{\lambda, m,n}$ and the first coordinate of our input is $k=1$. By 
Lemma~\ref{lem:Ptableaux_structure} (a) and (b) we know that $C$ must appear in the first column of $T$ directly above $A_j$ for some $j\in[m-1]$. Let $T'$ be $T$ but we switch $A_j$ and $A_m$. Because $A_j$ and $A_m$ have the same relations in $\cP_{m+1,n-1}$ certainly $T'\in \cT_{\lambda,m+1,n-1}$. Since  $C$ is now immediately above $A_m$  we can conclude by Lemma~\ref{lem:Ptableaux_structure} (b) that $T'\in \cT_{\lambda, m,n}$. We map
$$f_{\lambda}(1,T)= (j,T').$$
See Figure~\ref{fig:cases1and2} for an example. 
Note  in this case that the first coordinate of the output is at most $m-1$, which means outputs from Case 2 do not intersect with outputs from Case 1. Also note  in all outputs from this case that $C$ appears directly above $A_m$ in $T'$. This guarantees  that we are injective. We will show  in Case 3 that our outputs do not have $C$ immediately above $A_m$.

\begin{figure}

\begin{tikzpicture}
\begin{scope}[shift={(0,0)}]
\node at (-6,0) {$(2,$};
\node at (-3.5,0) {$)\mapsto(4,$};
\node at (-1.2,0) {$)$};
\node at (-5,0) {$
\begin{Young} 
$C$ & $B_2$ \cr
$A_3$ & $B_1$ \cr
$A_1$ & $B_3$ \cr
$A_2$ \cr
\end{Young}
$};
\node at (-2.1,0) {$
\begin{Young} 
$C$ & $B_2$ \cr
$A_3$ & $B_1$ \cr
$A_1$ & $B_3$ \cr
$A_2$ \cr
\end{Young}
$};
\end{scope}
\begin{scope}[shift={(7,0)}]
\node at (-6,0) {$(1,$};
\node at (-3.5,0) {$)\mapsto(2,$};
\node at (-1.2,0) {$)$};
\node at (-5,0) {$
\begin{Young} 
$C$ & $B_2$ \cr
$A_2$ & $B_1$ \cr
$A_1$ & $B_3$ \cr
$A_3$ \cr
\end{Young}
$};
\node at (-2.1,0) {$
\begin{Young} 
$C$ & $B_2$ \cr
$A_3$ & $B_1$ \cr
$A_1$ & $B_3$ \cr
$A_2$ \cr
\end{Young}
$};
\end{scope}
\end{tikzpicture}
\caption{We have two examples for $f_{\lambda}:[2]\times\cT_{\lambda,4,3}\rightarrow[4]\times\cT_{\lambda,3,4}$ where $\lambda = (2,2,2,1)$. On the left we have an example from Case 1 and on the right we have an example from Case 2. }
\label{fig:cases1and2}
\end{figure}

{\bf Case 3:} Say that we still have $T\notin \cT_{\lambda, m,n}$ but the first coordinate of our input is now $k = 2$. Again by
Lemma~\ref{lem:Ptableaux_structure}  (a) and (b) we know that $C$ is appears in the first column of $T$ directly above $A_j$ for some $j\in[m-1]$. This case will be more complicated, but we will be mapping 
$$f_{\lambda}(2,T)=(j,T'')$$
for some $P$-tableau $T''\in \cT_{\lambda,m,n}$ whose construction we describe next. 
We will describe the construction in steps: first what we will call an $\cA$-shift, then some column $\cB$-shifts. 

{\bf The $\cA$-shift:}
According to Lemma~\ref{lem:Ptableaux_structure} (c) we can decompose the first column of $T$ as in Figure~\ref{fig:def_ashift} where $\cA'$ and $\cA''$ are contiguous blocks of cells containing vertices from $\cA$ and $\cB'$ and $\cB''$ are contiguous blocks of cells containing vertices from $\cB$, any of which could possibly be empty.  We will shift the blocks of cells containing $\cB'$ and $C$ down below $A_j$ and $\cA''$ which we shift up as displayed in Figure~\ref{fig:def_ashift}. We will call this move the {\it $\cA$-shift} and the new tableau formed $T'$. 
We can see that column 1 in $T'$ satisfies the conditions necessary in order to be a $P$-tableau for $L_{m+1,n-1}$  by Lemma~\ref{lem:Ptableaux_structure} (c). 
However,  we are not guaranteed that $T'$ is a $P$-tableau for $L_{m+1,n-1}$ because of possible issues between columns 1 and  2. From Lemma~\ref{lem:Ptableaux_structure} (c) we know that $\cB'$, if nonempty, contains fillings $B_1, B_2, \ldots, B_s$ with subscripts increasing as we go up. So our $\cA$-shift moves around fillings from $\cA\cup \{C,B_1,B_2,\ldots, B_s\}$. Because all $B_1, B_2, \ldots, B_s$ are in column 1 we can conclude that all of $\cA\cup \{C,B_1,B_2,\ldots, B_s\}$ share the same relations with any possible filling from column 2 (except the relation between $B_s$ and $B_{s+1}$), so the $\cA$-shift preserves all properties we need  in order to be a $P$-tableau for $L_{m+1,n-1}$ except in the case where $B_s$ gets shifted down to be left of $B_{s+1}$. If that is not the case let $T'=T''$ and we are done. 

{\bf The column 2 $\cB$-shift:}
Now consider the unfortunate case where $B_s$ gets shifted down to be left of $B_{s+1}$. In this case we know in $T$ that $B_s$ is in column 1 row $r$ and $B_{s+1}$ is in column 2 below row $r$ in row $r+a$ where $a=\#(\cA''\cup\{A_j\})$. By  Lemma~\ref{lem:Ptableaux_structure} (d)  we can conclude that above and including $B_{s+1}$ in column 2 we have all of $B_{i}$ for $i\in[s+1,s+r+a]$ in a contiguous block of cells with subscripts increasing as we go up and $B_{s+r+a}$  in row 1. We will vertically cycle the block of cells containing $B_{i}$ for $i\in[s+1,s+r+a]$  so that $B_{s+1}$ is in row $r$ (immediately right of where $B_s$ was originally in $T$). Call this the {\it column 2 $\cB$-shift}. Using the last parts of Lemma~\ref{lem:Ptableaux_structure} (d)  we are guaranteed that the first two columns satisfying the conditions needed to be a $P$-tableau for $L_{m+1,n-1}$. If the resulting tableau happens to additionally be a $P$-tableau for $L_{m+1,n-1}$ then this $P$-tableau is our $T''$. 

{\bf The column 3 $\cB$-shift and further column $\cB$-shifts:}
Otherwise by similar reasons as before  to the right of $B_{s+r+a}$ after the column 2 $\cB$-shift we have $B_{s+r+a+1}$, which is in row $r+1$. Also similar to before, above and including $B_{s+r+a+1}$ we have all of $B_{i}$ for $i\in[s+r+a+1,s+2r+a+1]$ with subscripts increasing as we go up and $B_{s+2r+a+1}$ is in row 1. We will vertically cycle the block of cells containing $B_{i}$ for $i\in[s+r+a+1,s+2r+a+1]$ so that $B_{s+r+a+1}$ is in row $1$ (immediately right of where $B_{s+r+a}$ was originally in $T$). Call this the {\it column 3 $\cB$-shift}. We will continue doing these column $\cB$-shifts until we arrive at a $P$-tableau of $L_{m+1,n-1}$.   Note that it is straightforward to see that this will always  terminate successfully as follows. We can always do a $\cB$-shift unless some column $c$ only contains one cell, which is in the top row. Let this $c$ be minimal. In this case, the previous $\cB$-shift in column $c-1$ will have replaced the $B_i$ that was in its top row with some $B_{i'}$ where $i'<i$ and hence we will not need to perform a $\cB$-shift in column $c$. This $P$-tableau we created is $T''$. See Figure~\ref{fig:def_ashift} for an example. 

During our construction of $T''$ we have shown that $T''\in\cT_{\lambda,m+1,n-1}$. Because the filling immediately below $C$ in $T''$ is either non-existent or a filling from $\cB$ we know by Lemma~\ref{lem:Ptableaux_structure} (b) that $T''\in\cT_{\lambda,m,n}$, so our map is well defined. 
Recall that in this case we  mapped $f_{\lambda}(2,T)=(j,T'')$ where $j\leq m-1$, so we have in the output the filling $C$ in $T''$  not immediately above $A_m$. This means that Case 3 outputs do not intersect with those from Case 1 or Case 2.

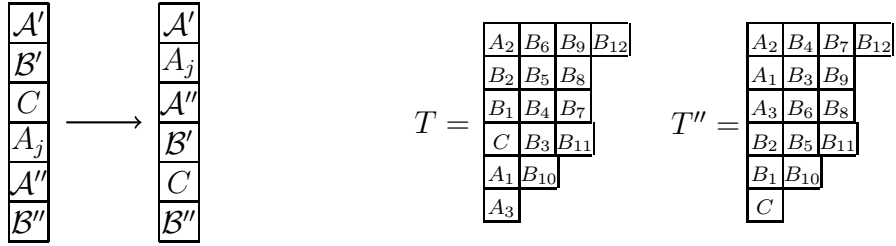
\begin{figure}
\centering 
\begin{tikzpicture}[scale= 1]
\node at (0,0) {$$\tableau{\mathcal{A}' \\ \mathcal{B}' \\ C \\ A_j \\ \mathcal{A}'' \\ \mathcal{B}''}$$}; 
\node at (2,0) {$$\tableau{\mathcal{A}' \\ A_j \\ \mathcal{A}'' \\ \mathcal{B}' \\ C \\ \mathcal{B}''}$$};
\draw[thick, ->] (0.5,0) -- (1.5,0);
\begin{scope}[shift={(7,0)}]
\node at (-1.5,0) {$T=$};
\node at (0,0) {$
\begin{Young} 
\tiny$A_2$ & \tiny$B_6$ & \tiny$B_9$ & \tiny$B_{12}$\cr
\tiny$B_2$ & \tiny$B_5$& \tiny$B_8$ \cr
\tiny$B_1$ & \tiny$B_4$& \tiny$B_7$ \cr
\tiny$C$& \tiny$B_3$& \tiny$B_{11}$ \cr
\tiny$A_1$& \tiny$B_{10}$ \cr
\tiny$A_3$ \cr
\end{Young}
$};
\end{scope}
\begin{scope}[shift={(10.5,0)}]
\node at (-1.5,0) {$T''=$};
\node at (0,0) {$
\begin{Young} 
\tiny$A_2$ & \tiny$B_4$ & \tiny$B_7$ & \tiny$B_{12}$\cr
\tiny$A_1$ & \tiny$B_3$& \tiny$B_9$ \cr
\tiny$A_3$ & \tiny$B_6$& \tiny$B_8$ \cr
\tiny$B_2$& \tiny$B_5$& \tiny$B_{11}$ \cr
\tiny$B_1$& \tiny$B_{10}$ \cr
\tiny$C$ \cr
\end{Young}
$};
\end{scope}

\end{tikzpicture}
\caption{On the left we illustrate the $\cA$-shift in general  by drawing only the first column. On the right we display $T$ and $T''$ in an example for Case 3 in the map $f_{\lambda}:[2]\times \cT_{\lambda,4,12}\rightarrow[4]\times\cT_{\lambda,3,13}$ where $\lambda = (4,3,3,3,2,1)$ and $f_{\lambda}(2,T)=(1,T'')$.}
\label{fig:def_ashift}
\end{figure}

Lastly, we only have to argue why Case 3 is injective. Case 3 is injective because if $f_{\lambda}(2,T)=(j,T'')$ comes from Case 3 we can recover $T$ from $(j,T'')$. Note that because of the $\cA$-shift $T''$ has its first column as follows reading from top to bottom:  a contiguous block of cells containing all of $\cA$ in some order,  a contiguous block of cells $\cB'$ with fillings from $\cB$,  $C$ and finally another contiguous block of cells $\cB''$ with fillings from $\cB$.  Because $j$ is specified we can split the top contiguous block of cells containing everything from  $\cA$ into $\cA'$, those above $A_j$, and $\cA''$, those below $A_j$. This allows us to undo the $\cA$-shift. Let $\tilde{T}'$ be $T''$ with the $\cA$-shift undone. Because of how we defined the column 2 $\cB$-shift, we know we did a  column 2 $\cB$-shift if there is some $B_s$ in column 1 followed immediately to its  right by $B_{s+1}$. Note that in $\tilde{T}'$ the occurrence of a $B_s$ in column 1 followed immediately to its right by $B_{s+1}$ will only happen because of a $\cB$-shift. If we can identify the block of cells we cycled in the column 2 $\cB$-shift then we will be able to undo this shift. This block will contain the cell $B_{s+1}$, the cell below $B_{s+1}$ filled with some $B_{s +t}$, $t>1$, and all other cells containing $B_{i}$, $i\in [s+l,s+t]$ for some smallest $l\leq t$ with subscripts decreasing consecutively as we go down. The block of  cells continues above $B_{s +1}$ up until we reach row 1 with subscripts  increasing consecutively as we go up. Now that we have identified all cells in the block that we cycled in the column 2 $\cB$-shift we can undo the column 2 $\cB$-shift by cycling vertically until $B_{s+1}$ is at the bottom. We can similarly undo all other column $\cB$-shifts, each of which will be indicated by some $B_{i}$ in column $c$ with a $B_{i+1}$ to its immediate right. Since we have recovered $T$ from $(j,T'')$ we have proven that  Case 3 is injective and can further conclude that $f_{\lambda}$ is injective.

\begin{theorem}\label{the:lollipop}
The family of {lollipop} graphs $\{L_{m,n}\}$ on $m+n=N$ vertices forms a chain in  $\cS_N$. {In particular, for $m\geq 3$ and $n\geq 0$,
$$L_{m-1,n+1}\geq_s L_{m,n}.$$}
{Hence, t}he path $P_N$ is the maximal element and the complete graph $K_N$ is the minimal element of the chain.
\end{theorem}

\section{Further  directions}
\label{sec:futher_directions}

One motivation behind setting up these posets was to study $e$-positivity and Schur-positivity of chromatic symmetric functions. More is known about Schur-positivity of chromatic symmetric functions since Gasharov~\cite{Gash} has proven that incomparability graphs of $(3+1)$-free posets are Schur-positive, but it has yet to be proven that the same class of graphs are all $e$-positive. Guay-Paquet~\cite{GP} has reduced the question to only needing to show that incomparability graphs of $(2+2)$ and $(3+1)$-free posets are $e$-positive, a family of graphs better known as unit interval graphs. Since we have proven that $G$ is $e$-positive if and only if $G\geq_e K_n$ in $\cE_n$ in Theorem~\ref{thm:G>K_n} the poset $\cE_n$ gives us another approach to $e$-positivity. We can now show $G$ is $e$-positive by finding and proving a sequence of inequalities $G\geq_e G_1\geq_e G_2\geq_e \cdots \geq_e G_l\geq_e K_n$. For example the inequalities proven in Theorem~\ref{thm:lollipop_e-chain} prove that lollipop graphs are $e$-positive, although, lollipops have been proven to be $e$-positive previously by other methods~\cite{DvW,GebSag}. 

Below we conjecture many inequalities between connected unit interval graphs in $\cE_n$  and in $\cS_n$. These inequalities can be placed in series to show any connected unit interval graph is $e$-positive and Schur-positive, though Gasharov~\cite{Gash} has proven all unit interval graphs are Schur-positive. First we need a description of unit interval graphs.  There are many equivalent definitions for unit interval graphs with some equivalences proven in~\cite{E18}. Here we will describe  {\it unit interval graphs}  on vertices in $[n]$  using a weakly-increasing sequence ${\bf m}=(m_1,m_2,\ldots ,m_{n-1})$ where $i\leq m_i\leq n$ for all $i\in[n-1]$. The graph will have an edge between $a$ and $b$ whenever $a,b\in [i,m_i]$ for some $i$.

\begin{conjecture}
Let $G$ be a connected unit interval graph defined by the weakly-increasing sequence ${\bf m}=(m_1,m_2,\ldots ,m_{n-1})$, $i\leq m_i\leq n$ for all $i\in[n-1]$. 
Let $G'$ be the  unit interval graph defined by the sequence ${\bf m}'=(m_1+1,\ldots, m_r+1,m_{r+1},\ldots ,m_{n-1})$ where $m_r<m_{r+1}$ is the first increase.
Then,
$$G\geq_e G' \text{ and } G\geq_s G'.$$
\label{conj}
\end{conjecture}

Using deletion-contraction we can compute the chromatic polynomial for all unit interval graphs and the coefficients $[e_{(n)}]X_G$ and $[s_{(1^n)}]X_G$.
\begin{proposition}
Let $G$ be a  unit interval graph defined by the weakly-increasing sequence ${\bf m}=(m_1,m_2,\ldots ,m_{n-1})$, $i\leq m_i\leq n$ for all $i\in[n-1]$. Then, 
$$\chi_G(k)=k\prod_{i=1}^{n-1}(k-(m_i-i)).$$
Also, 
$$[e_{(n)}]X_G=n\prod_{i=1}^{n-1}(m_i-i) \text{ and }[s_{(1^n)}]X_G=\prod_{i=1}^{n-1}(m_i-i+1).$$
\end{proposition}

\begin{proof}
Let $G$ be a  unit interval graph defined by the weakly-increasing sequence ${\bf m}=(m_1,m_2,\ldots ,m_{n-1})$, $i\leq m_i\leq n$ for all $i\in[n-1]$. The formula for the chromatic polynomial will follow from deletion-contraction and induction. Our base case is when $n=1$ and ${\bf m}$ is an empty list. In this case $\chi_{G}=k$, which matches the formula. Using deletion-contraction repeatedly on all edges connected to vertex 1 we get 
$$\chi_G=\chi_{G'}-(m_1-1)\chi_{G''},$$
where $G'$ and $G''$ are the graphs associated to ${\bf m}'=(1,m_2,\ldots ,m_{n-1})$ and ${\bf m}''=(m_2-1,m_3-1,\ldots ,m_{n-1}-1)$, respectively. Note that $G'$ is $G''$, but with an additional disjoint vertex. This means $\chi_{G'}=k\chi_{G''}$. Then, 
$$\chi_G=\chi_{G''}(k-(m_1-1)),$$
so by induction we have our formula.

By equation~\eqref{eq:e_n_coeff} we get that $[e_{(n)}]X_G=n\prod_{i=1}^{n-1}(m_i-i)$.
By equation~\eqref{eq:s_1^n_coeff} we get that $[s_{(1^n)}]X_G=\prod_{i=1}^{n-1}(m_i-i+1)$. 
\end{proof}

\begin{remark}Since we can form a sequence of connected unit interval graphs from any connected unit interval graph to the  complete graph, as described in Conjecture~\ref{conj},  proving Conjecture~\ref{conj} would imply that all unit interval graphs are $e$-positive. The conjecture has been confirmed up until $n=7$.
\end{remark}


\begin{thebibliography}{9}

\bibitem{Jose2+1} Aliste-Prieto, J., de Mier, A. and Zamora, J., 
On trees with the same restricted 
$U$-polynomial and the Prouhet-Tarry-Escott problem, 
\emph{Discrete Math.} 340, 1435--1441, (2017).

\bibitem{Jose2} Aliste-Prieto, J. and Zamora, J., 
Proper caterpillars are distinguished by their chromatic symmetric function, 
\emph{Discrete Math.} 315, 158--164, (2014).

\bibitem{B12} Birkhoff, G., 
A determinant formula for the number of ways of coloring a map, 
\emph{Ann. of Math.} 14, 43--46, (1912).

\bibitem{BW} Brightwell, G. and Winkler, P.,  
Maximum hitting time for random walks on graphs, 
\emph{Random Structures Algorithms} 1, 263--276, (1990).

\bibitem{ChoHuh} Cho, S. and Huh, J., 
On $e$-positivity and $e$-unimodality of chromatic quasisymmetric functions, 
\emph{SIAM J. Discrete Math.} 33(4), 2286--2315, (2019).

\bibitem{DFvW}
Dahlberg, S.,  Foley, A. and van Willigenburg, S., 
Resolving Stanley's $e$-positivity of claw-contractible-free graphs, 
\emph{arXiv:1703.05770v1}, 1--22, (2017). 

\bibitem{DvW}
Dahlberg, S. and van Willigenburg, S., 
Lollipop and lariat symmetric functions, 
\emph{SIAM J. Discrete Math.} 32(2), 1029--1039, (2018). 

\bibitem{E18}
Ellzey, B., 
On the chromatic quasisymmetric functions of directed graphs, 
{Thesis (Ph.D.)--University of Miami},
\emph{ProQuest LLC}, 1--146, (2018). 


\bibitem{Feige} Feige, U., 
A tight upper bound on the cover time for random walks on graphs, 
\emph{Random Structures Algorithms} 6, 51--54, (1995).

\bibitem{Gasharov} Gasharov, V., 
Incomparability graphs of $(3 + 1)$-free posets are $s$-positive, 
\emph{Discrete Math.} 157, 193--197, (1996).

\bibitem{Gash} Gasharov, V., 
On Stanley's chromatic symmetric function and clawfree graphs, 
\emph{Discrete Math.} 205, 229--234, (1999).

\bibitem{GebSag} Gebhard, D. and Sagan, B.,  
A chromatic symmetric function in noncommuting variables, 
\emph{J. Algebraic Combin.} 13, 227--255, (2001). 

\bibitem{GZ83} Greene, C.  and  Zaslavsky, T., 
On the interpretation of Whitney numbers through arrangements of hyperplanes, zonotopes, non-Radon partitions, and orientations of graphs, 
\emph{Trans. Amer. Math. Soc.} 280, 97--126, (1996).

\bibitem{GP} Guay-Paquet, M., 
A modular relation for the chromatic symmetric functions of $(3+1)$-free posets, 
\emph{arXiv:1306.2400}, 1--10, (2013). 

\bibitem{HP} Harada, M. and Precup, M., 
The cohomology of abelian Hessenberg varieties and the Stanley-Stembridge conjecture, 
\emph{S\'{e}m. Lothar. Combin.} 80B, 1--12, (2018). 

\bibitem{Jonasson} Jonasson, J., 
Lollipop graphs are extremal for commute times, 
\emph{Random Structures  Algorithms} 16, 131--142, (2000).

\bibitem{Kaliszwecki}
Kaliszewski R., 
Hook coefficients of chromatic functions, 
\emph{J. Comb.} 6(3), 327--337, (2015).

\bibitem{M79}
Macdonald, I., 
Symmetric functions and Hall polynomials,
\emph{Oxford University Press}, edition 2, (2015).

\bibitem{MMW} Martin, J., Morin, M. and Wagner, J., 
On distinguishing trees by their chromatic symmetric functions, 
\emph{J. Combin. Theory Ser. A} 115, 237--253, (2008).

\bibitem{Orellana} Orellana, R. and Scott, G., 
Graphs with equal chromatic symmetric function, 
\emph{Discrete Math.} 320, 1--14, (2014).

\bibitem{S01}
Sagan, B., 
The Symmetric Group,  
\emph{Springer-Verlag}, edition 2, (2001).  

\bibitem{SW} Shareshian, J. and Wachs, M., 
Chromatic quasisymmetric functions, 
\emph{Adv. Math.} 295, 497--551, (2016).

\bibitem{Stan73} Stanley, R., 
Acyclic orientations of graphs, 
\emph{Discrete Math} 5:2, 171--178, (1973).

\bibitem{Stan95} Stanley, R., 
A symmetric function generalization of the chromatic polynomial of a graph, 
\emph{Adv. Math.} 111, 166--194 (1995).

\bibitem{stanley1}
Stanley, R.,  
Enumerative Combinatorics Vol. 1, 
\emph{Cambridge University Press}, edition 2, (2012). 

\bibitem{stanley2}
Stanley, R., 
Enumerative Combinatorics Vol. 2, 
\emph{Cambridge University Press}, (1999).

\bibitem{StanStem} Stanley, R. and Stembridge, J., 
On immanant of Jacobi-Trudi matrices and permutations with restricted position, 
\emph{J. Combin. Theory Ser. A} 62, 261--279, (1993).

\bibitem{SWW} Sundquist, T., Wagner, D. and West, J., 
A Robinson-Schensted algorithm for a class of partial orders, 
\emph{J. Combin. Theory Ser. A} 79, 36--52, (1997).

\bibitem{Wolfe} Wolfe, M.,
 Symmetric chromatic functions, 
 \emph{Pi Mu Epsilon J.} 10, 643--757, (1998).

\end{thebibliography}
\end{document}